\documentclass[a4paper,reqno]{amsart}

\usepackage{graphicx}
\usepackage{amsmath}
\usepackage{amssymb}
\input xy
\xyoption{all}

\usepackage{color}

\numberwithin{equation}{section}

\theoremstyle{plain}

\newtheorem{thm}{Theorem}[section]
\newtheorem{cor}[thm]{Corollary}
\newtheorem{lem}[thm]{Lemma}
\newtheorem{prop}[thm]{Proposition}
\newtheorem{fact}[thm]{Fact}

\theoremstyle{defn}
\newtheorem{defn}[thm]{Definition}
\newtheorem{exm}[thm]{Example}

\theoremstyle{remark}
\newtheorem{rem}[thm]{Remark}

\newdir{ >}{{}*!/-5pt/\dir{>}}

\renewcommand{\mod}{\operatorname{mod}\nolimits}

\newcommand{\add}{\operatorname{add}\nolimits}

\newcommand{\id}{\operatorname{id}\nolimits}

\newcommand{\op}{\operatorname{op}\nolimits}

\newcommand{\ra}{\rightarrowtail}
\newcommand{\ta}{\twoheadrightarrow}

\newcommand{\Cone}{\operatorname{Cone}\nolimits}
\newcommand{\CoCone}{\operatorname{CoCone}\nolimits}

\newcommand{\M}{\mathcal M}

\newcommand{\B}{\mathcal B}
\newcommand{\uB}{\underline{\B}}
\newcommand{\oB}{\overline{\B}}
\newcommand{\U}{\mathcal U}
\newcommand{\V}{\mathcal V}
\newcommand{\W}{\mathcal W}
\newcommand{\h}{\mathcal H}
\newcommand{\s}{\mathcal S}
\newcommand{\T}{\mathcal T}

\newcommand{\K}{\mathcal K}

\newcommand{\Y}{\mathcal Y}

\newcommand{\C}{\mathcal C}

\newcommand{\EE}{\mathbb E}
\newcommand{\STUV}{((\s,\T),(\U,\V))}

\newcommand{\svecv}[2]{\left(\begin{smallmatrix}
      #1 \\
      #2
    \end{smallmatrix}\right)}

\newcommand{\svech}[2]{\left(\begin{smallmatrix}
      #1 & #2
\end{smallmatrix}\right)}

\renewcommand{\emph}{\textit}
\renewcommand{\phi}{\varphi}

\begin{document}

\title{Hearts of twin Cotorsion pairs on extriangulated categories}
\author{Yu Liu}
\email{liu.yu@math.uu.se} % foptionalf
\address{Matematiska institutionen, Uppsala Universitet, Box 480, 751 06 Uppsala, Sweden}

\author{Hiroyuki Nakaoka}
\email{nakaoka@sci.kagoshima-u.ac.jp} % foptionalf
\address{Research and Education Assembly, Science and Engineering Area, Research Field in Science, Kagoshima University, 1-21-35 Korimoto, Kagoshima, 890-0065 Japan}
\thanks{The first author is supported by Knut and Alice Wallenberg foundation program for mathematics.}
\thanks{The second author is supported by JSPS KAKENHI Grant Numbers 25800022.}
\thanks{The authors wish to thank Professor Martin Herschend, Doctor Panyue Zhou and Professor Bin Zhu for their helpful advices.}

\begin{abstract}
In this article, we study the {\it heart} of a cotorsion pairs on an exact category and a triangulated category in a unified meathod, by means of the notion of an {\it extriangulated category}. We prove that the heart is abelian, and construct a cohomological functor to the heart. If the extriangulated category has enough projectives, this functor gives an equivalence between the heart and the category of coherent functors over the {\it coheart} modulo projectives. We also show how an {\it$n$-cluster tilting subcategory} of an extriangulated category gives rise to a family of cotorsion pairs with equivalent hearts.
\end{abstract}

\maketitle

\tableofcontents

\section{Introduction and Preliminaries}

The notion of a {\it cotorsion pair} goes back to \cite{S}. It has been defined originally in module categories, and then in abelian and exact categories. Cotorsion pairs give a perspective of resolutions in the given category, and are related with other homological notions which are interesting in the category theory, or in the representation theory. For example, a pair of cotorsion pairs satisfying some conditions, which we call a {\it twin cotorsion pair}, has close relationship with model structure called {\it abelian model structure} (\cite{Ho1},\cite{Ho2}) and {\it exact model structure} \cite{G}.

If one considers a cotorsion pair in a triangulated category, it becomes essentially the same as the notion of a torsion theory in the sense of \cite{IY}, which has been introduced in the recent development of the theory of mutation, and of higher Auslander-Reiten theory. It gives a simultaneous generalization of $t$-structure \cite{BBD}, cluster tilting subcategory \cite{KZ}, co-$t$-structure, and is also related to notions such as $n$-cluster tilting subcategory.
Quite often, an argument on cotorsion pairs which works in an exact category has its counterpart in a triangulated category through some appropriate modification, and vice versa. As for the relation with model structures mentioned above, its analog in a triangulated category has been given by \cite{Y}.

Recently, the notion of an {\it extriangulated category} was introduced in \cite{NP}, which is a simultaneous generalization of exact category and triangulated category. This enables us to unify the arguments on cotorsion pairs, as demonstrated in \cite{NP} for the relation with model structures.
Amongst several notions related to cotorsion pairs, this article is especially devoted to a unified treatment of the {\it hearts} of cotorsion pairs, by means of an extriangulated category.

Historically, the heart of a t-structure has been introduced in \cite{BBD} to define perverse sheaves. More recently, the quotient category of a triangulated category by a cluster tilting subcategory has been used in the theory of cluster category \cite{KZ,DL}. As a simultaneous generalization of them, the notions of hearts of cotorsion pairs were introduced in \cite{N1} for a triangulated category, and in \cite{L1} for an exact category. In both cases, the hearts are shown to become abelian categories, and they have similar properties.

In the latter part of this section, we briefly recall the definition and some basic properties of an extriangulated category, which will be used throughout this article.
Although our main concern is about cotorsion pairs, we also deal with {{\it twin} cotorsion pairs in section \ref{Section_TwinHeart}. The definition of the heart and its basic properties are given there. Especially, we will show that the heart of a twin cotorsion pair is always semi-abelian (Theorem \ref{C4.3}). Since we regard a cotorsion pair as a special case of a twin cotorsion pair (Definition \ref{Cctp}), which we call a {\it single} cotorsion pair when we emphasize, the results in section \ref{Section_TwinHeart} will be used in the proceeding sections.

From section \ref{Section_Heart}, we mainly deal with single cotorsion pairs. We specialize the results in section \ref{Section_TwinHeart}, to show that the heart of a cotorsion pair becomes abelian (Theorem \ref{C18}), and that the associated functor to the heart becomes {\it cohomological} (Theorem \ref{8.11}). These unify the results in \cite{N1}, \cite{L1} and \cite{AN}, \cite{L2} respectively. We remark that we do not need the assumption of enough projectivity nor injectivity until this section, which gives a slight improvement of \cite{L1} even in the case of an exact category. In the latter part, we also give a criterion for the hearts of cotorsion pairs to be equivalent compatibly with the associated cohomological functors (Proposition \ref{Prop_HeartEq}). It turns out that the {\it kernel} $\K$ of a cotorsion pair $(\U,\V)$ plays an important role.

From section \ref{Section_ProjectiveHeart}, we assume that the extriangulated category has enough projectives. Under this assumption, we show that a pair associated to $\K$ forms a cotorsion pair if and only if the heart of $(\U,\V)$ has enough projectives of certain type (Theorem \ref{HCoH}). We also give an equivalence between the heart and the category of coherent functors over the {\it coheart} modulo projectives (Proposition \ref{Propmod}).  The notion of a coheart is a generalization of that for a co-$t$-structure.

If we also assume the existence of enough injectives, we obtain a natural notion of higher extensions. This allows us to define {\it $n$-cluster tilting subcategories} of an extriangulated category. It gives a simultaneous generalization of the case of a triangulated category, and of an exact category. In section \ref{Section_nCluster}, we show how an $n$-cluster tilting subcategory induces a sequence of cotorsion pairs (Theorem \ref{nCT}). Combining the results in the preceeding sections, we will see that all these have equivalent hearts, which are equivalent to some category of coherent functors (Corollary \ref{CornCT}).

\bigskip

Let us briefly recall the definition and basic properties of extriangulated categories from \cite{NP}. Throughout this article, let $\B$ be an additive category.

\begin{defn}\label{DefExtension}
Suppose $\B$ is equipped with an additive bifunctor $\mathbb{E}\colon\B^\mathrm{op}\times\B\to\mathit{Ab}$. For any pair of objects $A,C\in\B$, an element $\delta\in\mathbb{E}(C,A)$ is called an {\it $\mathbb{E}$-extension}.
\end{defn}

\begin{rem}
Let $\delta\in\EE(C,A)$ be any $\mathbb{E}$-extension. By the functoriality, for any $a\in\B(A,A^{\prime})$ and $c\in\B(C^{\prime},C)$, we have $\mathbb{E}$-extensions
\[ \mathbb{E}(C,a)(\delta)\in\mathbb{E}(C,A^{\prime})\ \ \text{and}\ \ \mathbb{E}(c,A)(\delta)\in\mathbb{E}(C^{\prime},A). \]
We abbreviately denote them by $a_{\ast}\delta$ and $c^{\ast}\delta$.
In this terminology, we have
\[ \mathbb{E}(c,a)(\delta)=c^{\ast} a_{\ast}\delta=a_{\ast} c^{\ast}\delta \]
in $\mathbb{E}(C^{\prime},A^{\prime})$.
\end{rem}

\begin{defn}\label{DefMorphExt}
Let $\delta\in\EE(C,A),\delta^{\prime}\in\EE(C^{\prime},A^{\prime})$ be any pair of $\mathbb{E}$-extensions. A {\it morphism} $(a,c)\colon\delta\to\delta^{\prime}$ of $\mathbb{E}$-extensions is a pair of morphisms $a\in\B(A,A^{\prime})$ and $c\in\B(C,C^{\prime})$ in $\B$, satisfying the equality
\[ a_{\ast}\delta=c^{\ast}\delta^{\prime}. \]
We simply denote it as $(a,c)\colon\delta\to\delta^{\prime}$.
\end{defn}

\begin{defn}\label{DefSplitExtension}
For any $A,C\in\B$, the zero element $0\in\mathbb{E}(C,A)$ is called the {\it split $\mathbb{E}$-extension}.
\end{defn}

\begin{defn}\label{DefSumExtension}
Let $\delta=(A,\delta,C),\delta^{\prime}=(A^{\prime},\delta^{\prime},C^{\prime})$ be any pair of $\mathbb{E}$-extensions. Let
\[ C\overset{\iota_C}{\longrightarrow}C\oplus C^{\prime}\overset{\iota_{C^{\prime}}}{\longleftarrow}C^{\prime} \]
and
\[ A\overset{p_A}{\longleftarrow}A\oplus A^{\prime}\overset{p_{A^{\prime}}}{\longrightarrow}A^{\prime} \]
be coproduct and product in $\B$, respectively. Remark that, by the additivity of $\mathbb{E}$, we have a natural isomorphism
\[ \mathbb{E}(C\oplus C^{\prime},A\oplus A^{\prime})\cong \mathbb{E}(C,A)\oplus\mathbb{E}(C,A^{\prime})\oplus\mathbb{E}(C^{\prime},A)\oplus\mathbb{E}(C^{\prime},A^{\prime}). \]

Let $\delta\oplus\delta^{\prime}\in\mathbb{E}(C\oplus C^{\prime},A\oplus A^{\prime})$ be the element corresponding to $(\delta,0,0,\delta^{\prime})$ through this isomorphism. This is the unique element which satisfies
\begin{eqnarray*}
\mathbb{E}(\iota_C,p_A)(\delta\oplus\delta^{\prime})=\delta&,&\mathbb{E}(\iota_C,p_{A^{\prime}})(\delta\oplus\delta^{\prime})=0,\\
\mathbb{E}(\iota_{C^{\prime}},p_A)(\delta\oplus\delta^{\prime})=0&,&\mathbb{E}(\iota_{C^{\prime}},p_{A^{\prime}})(\delta\oplus\delta^{\prime})=\delta^{\prime}.
\end{eqnarray*}
\end{defn}

\begin{defn}\label{DefSqEquiv}
Let $A,C\in\B$ be any pair of objects. Sequences of morphisms in $\B$
\[ A\overset{x}{\longrightarrow}B\overset{y}{\longrightarrow}C\ \ \text{and}\ \ A\overset{x^{\prime}}{\longrightarrow}B^{\prime}\overset{y^{\prime}}{\longrightarrow}C \]
are said to be {\it equivalent} if there exists an isomorphism $b\in\B(B,B^{\prime})$ which makes the following diagram commutative.
\[
\xy
(-16,0)*+{A}="0";
(3,0)*+{}="1";
(0,8)*+{B}="2";
(0,-8)*+{B^{\prime}}="4";
(-3,0)*+{}="5";
(16,0)*+{C}="6";
{\ar^{x} "0";"2"};
{\ar^{y} "2";"6"};
{\ar_{x^{\prime}} "0";"4"};
{\ar_{y^{\prime}} "4";"6"};
{\ar^{b}_{\cong} "2";"4"};
\endxy
\]

We denote the equivalence class of $A\overset{x}{\longrightarrow}B\overset{y}{\longrightarrow}C$ by $[A\overset{x}{\longrightarrow}B\overset{y}{\longrightarrow}C]$.
\end{defn}

\begin{defn}\label{DefAddSeq}
$\ \ $
\begin{enumerate}
\item[(1)] For any $A,C\in\B$, we denote as
\[ 0=[A\overset{\Big[\raise1ex\hbox{\leavevmode\vtop{\baselineskip-8ex \lineskip1ex \ialign{#\crcr{$\scriptstyle{1}$}\crcr{$\scriptstyle{0}$}\crcr}}}\Big]}{\longrightarrow}A\oplus C\overset{[0\ 1]}{\longrightarrow}C]. \]

\item[(2)] For any $[A\overset{x}{\longrightarrow}B\overset{y}{\longrightarrow}C]$ and $[A^{\prime}\overset{x^{\prime}}{\longrightarrow}B^{\prime}\overset{y^{\prime}}{\longrightarrow}C^{\prime}]$, we denote as
\[ [A\overset{x}{\longrightarrow}B\overset{y}{\longrightarrow}C]\oplus [A^{\prime}\overset{x^{\prime}}{\longrightarrow}B^{\prime}\overset{y^{\prime}}{\longrightarrow}C^{\prime}]=[A\oplus A^{\prime}\overset{x\oplus x^{\prime}}{\longrightarrow}B\oplus B^{\prime}\overset{y\oplus y^{\prime}}{\longrightarrow}C\oplus C^{\prime}]. \]
\end{enumerate}
\end{defn}

\begin{defn}\label{DefRealization}
Let $\mathfrak{s}$ be a correspondence which associates an equivalence class $\mathfrak{s}(\delta)=[A\overset{x}{\longrightarrow}B\overset{y}{\longrightarrow}C]$ to any $\mathbb{E}$-extension $\delta\in\mathbb{E}(C,A)$. This $\mathfrak{s}$ is called a {\it realization} of $\mathbb{E}$, if it satisfies the following condition $(\ast)$. In this case, we say that sequence $A\overset{x}{\longrightarrow}B\overset{y}{\longrightarrow}C$ {\it realizes} $\delta$, whenever it satisfies $\mathfrak{s}(\delta)=[A\overset{x}{\longrightarrow}B\overset{y}{\longrightarrow}C]$.
\begin{itemize}
\item[$(\ast)$] Let $\delta\in\mathbb{E}(C,A)$ and $\delta^{\prime}\in\mathbb{E}(C^{\prime},A^{\prime})$ be any pair of $\mathbb{E}$-extensions, with $\mathfrak{s}(\delta)=[A\overset{x}{\longrightarrow}B\overset{y}{\longrightarrow}C],\, \mathfrak{s}(\delta^{\prime})=[A^{\prime}\overset{x^{\prime}}{\longrightarrow}B^{\prime}\overset{y^{\prime}}{\longrightarrow}C^{\prime}]$.
Then, for any morphism $(a,c)\colon\delta\to\delta^{\prime}$, there exists $b\in\B(B,B^{\prime})$ which makes the following diagram commutative.
\begin{equation}\label{MorphRealize}
\xy
(-12,6)*+{A}="0";
(0,6)*+{B}="2";
(12,6)*+{C}="4";
(-12,-6)*+{A^{\prime}}="10";
(0,-6)*+{B^{\prime}}="12";
(12,-6)*+{C^{\prime}}="14";
{\ar^{x} "0";"2"};
{\ar^{y} "2";"4"};
{\ar_{a} "0";"10"};
{\ar^{b} "2";"12"};
{\ar^{c} "4";"14"};
{\ar_{x^{\prime}} "10";"12"};
{\ar_{y^{\prime}} "12";"14"};
\endxy
\end{equation}
\end{itemize}
In the above situation, we say that the triplet $(a,b,c)$ {\it realizes} $(a,c)$.
\end{defn}

\begin{defn}\label{DefAdditiveRealization}
Let $\B,\mathbb{E}$ be as above. A realization of $\mathbb{E}$ is said to be {\it additive}, if it satisfies the following conditions.
\begin{itemize}
\item[{\rm (i)}] For any $A,C\in\B$, the split $\mathbb{E}$-extension $0\in\mathbb{E}(C,A)$ satisfies
\[ \mathfrak{s}(0)=0. \]
\item[{\rm (ii)}] For any pair of $\mathbb{E}$-extensions $\delta\in\EE(C,A)$ and $\delta^{\prime}\in\EE(C^{\prime},A^{\prime})$,
\[ \mathfrak{s}(\delta\oplus\delta^{\prime})=\mathfrak{s}(\delta)\oplus\mathfrak{s}(\delta^{\prime}) \]
holds.
\end{itemize}
\end{defn}

\begin{defn}\label{DefExtCat}$($\cite[Definition 2.12]{NP}$)$
A triplet $(\B,\mathbb{E},\mathfrak{s})$ is called an {\it extriangulated category} if it satisfies the following conditions.
\begin{itemize}
\item[{\rm (ET1)}] $\mathbb{E}\colon\B^{\mathrm{op}}\times\B\to\mathit{Ab}$ is an additive bifunctor.
\item[{\rm (ET2)}] $\mathfrak{s}$ is an additive realization of $\mathbb{E}$.
\item[{\rm (ET3)}] Let $\delta\in\mathbb{E}(C,A)$ and $\delta^{\prime}\in\mathbb{E}(C^{\prime},A^{\prime})$ be any pair of $\mathbb{E}$-extensions, realized as
\[ \mathfrak{s}(\delta)=[A\overset{x}{\longrightarrow}B\overset{y}{\longrightarrow}C],\ \ \mathfrak{s}(\delta^{\prime})=[A^{\prime}\overset{x^{\prime}}{\longrightarrow}B^{\prime}\overset{y^{\prime}}{\longrightarrow}C^{\prime}]. \]
For any commutative square
\begin{equation}\label{SquareForET3}
\xy
(-12,6)*+{A}="0";
(0,6)*+{B}="2";
(12,6)*+{C}="4";
(-12,-6)*+{A^{\prime}}="10";
(0,-6)*+{B^{\prime}}="12";
(12,-6)*+{C^{\prime}}="14";
{\ar^{x} "0";"2"};
{\ar^{y} "2";"4"};
{\ar_{a} "0";"10"};
{\ar^{b} "2";"12"};
{\ar_{x^{\prime}} "10";"12"};
{\ar_{y^{\prime}} "12";"14"};
\endxy
\end{equation}
in $\B$, there exists a morphism $(a,c)\colon\delta\to\delta^{\prime}$ satisfying $cy=y^{\prime}b$.
\item[{\rm (ET3)$^{\mathrm{op}}$}] Let $\delta\in\mathbb{E}(C,A)$ and $\delta^{\prime}\in\mathbb{E}(C^{\prime},A^{\prime})$ be any pair of $\mathbb{E}$-extensions, realized by
\[ A\overset{x}{\longrightarrow}B\overset{y}{\longrightarrow}C\ \ \text{and}\ \ A^{\prime}\overset{x^{\prime}}{\longrightarrow}B^{\prime}\overset{y^{\prime}}{\longrightarrow}C^{\prime} \]
respectively.
For any commutative square
\[
\xy
(-12,6)*+{A}="0";
(0,6)*+{B}="2";
(12,6)*+{C}="4";
(-12,-6)*+{A^{\prime}}="10";
(0,-6)*+{B^{\prime}}="12";
(12,-6)*+{C^{\prime}}="14";
{\ar^{x} "0";"2"};
{\ar^{y} "2";"4"};
{\ar_{b} "2";"12"};
{\ar^{c} "4";"14"};
{\ar_{x^{\prime}} "10";"12"};
{\ar_{y^{\prime}} "12";"14"};
\endxy
\]
in $\B$, there exists a morphism $(a,c)\colon\delta\to\delta^{\prime}$ satisfying $bx=x^{\prime}a$.
\item[{\rm (ET4)}] Let $\delta\in\EE(D,A)$ and $\delta^{\prime}\in\EE(F,B)$ be $\mathbb{E}$-extensions realized by
\[ A\overset{f}{\longrightarrow}B\overset{f^{\prime}}{\longrightarrow}D\ \ \text{and}\ \ B\overset{g}{\longrightarrow}C\overset{g^{\prime}}{\longrightarrow}F \]
respectively. Then there exist an object $E\in\B$, a commutative diagram
\begin{equation}\label{DiagET4}
\xy
(-21,7)*+{A}="0";
(-7,7)*+{B}="2";
(7,7)*+{D}="4";
(-21,-7)*+{A}="10";
(-7,-7)*+{C}="12";
(7,-7)*+{E}="14";
(-7,-21)*+{F}="22";
(7,-21)*+{F}="24";
{\ar^{f} "0";"2"};
{\ar^{f^{\prime}} "2";"4"};
{\ar@{=} "0";"10"};
{\ar_{g} "2";"12"};
{\ar^{d} "4";"14"};
{\ar_{h} "10";"12"};
{\ar_{h^{\prime}} "12";"14"};
{\ar_{g^{\prime}} "12";"22"};
{\ar^{e} "14";"24"};
{\ar@{=} "22";"24"};
\endxy
\end{equation}
in $\B$, and an $\mathbb{E}$-extension $\delta^{\prime\prime}\in\mathbb{E}(E,A)$ realized by $A\overset{h}{\longrightarrow}C\overset{h^{\prime}}{\longrightarrow}E$, which satisfy the following compatibilities.
\begin{itemize}
\item[{\rm (i)}] $D\overset{d}{\longrightarrow}E\overset{e}{\longrightarrow}F$ realizes $f^{\prime}_{\ast}\delta^{\prime}$,
\item[{\rm (ii)}] $d^{\ast}\delta^{\prime}r=\delta$,

\item[{\rm (iii)}] $f_{\ast}\delta^{\prime}r=e^{\ast}\delta^{\prime}$.
\end{itemize}

\item[{\rm (ET4)$^{\mathrm{op}}$}] Dual of {\rm (ET4)}.
\end{itemize}
\end{defn}

\begin{exm}\label{Example1}
Exact categories and triangulated categories are extriangulated categories. See \cite{NP} for the detail.
\end{exm}

We use the following terminology.
\begin{defn}\label{DefTermExact1}
Let $(\B,\mathbb{E},\mathfrak{s})$ be a triplet satisfying {\rm (ET1)} and {\rm (ET2)}.
\begin{enumerate}
\item[(1)] A sequence $A\overset{x}{\longrightarrow}B\overset{y}{\longrightarrow}C$ is called an {\it conflation} if it realizes some $\mathbb{E}$-extension $\delta\in\mathbb{E}(C,A)$. In this article, we write the conflation as $A\overset{x}{\ra}B\overset{f}{\ta}C$.
\item[(2)] A morphism $f\in\B(A,B)$ is called an {\it inflation} if it admits some conflation $A\overset{f}{\ra}B\ta C$.
\item[(3)] A morphism $f\in\B(A,B)$ is called an {\it deflation} if it admits some conflation $K\ra A\overset{f}{\ta}B$.
\end{enumerate}
\end{defn}

\begin{defn}
Let $(\B,\mathbb{E},\mathfrak{s})$ as in Definition \ref{DefTermExact1}.
\begin{enumerate}
\item An object $P\in\B$ is called {\it projective} if it satisfies $\EE(P,\B)=0$. We denote the subcategory of projective objects by $\mathcal P\subseteq\B$. We say that $\B$ {\it has enough projectives} if any object $B\in\B$ admits a deflation $P\to B$ from some $P\in\mathcal P$.
\item Dually, an object $I\in\B$ is called {\it injective} if it satisfies $\EE(\B,I)=0$. We denote the subcategory of injective objects by $\mathcal I\subseteq\B$. We say that $\B$ {\it has enough injectives} if any object $B\in\B$ admits a inflation $B\to I$ to some $I\in\mathcal I$.
\end{enumerate}
\end{defn}

\begin{defn}\label{DefTermExact2}
Let $(\B,\mathbb{E},\mathfrak{s})$ be a triplet satisfying {\rm (ET1)} and {\rm (ET2)}.
\begin{enumerate}
\item[(1)] If a conflation $A\overset{x}{\ra}B\overset{y}{\ta}C$ realizes $\delta\in\mathbb{E}(C,A)$, we call the pair $(A\overset{x}{\ra}B\overset{y}{\ta}C,\delta)$ an {\it $\EE$-triangle}, and write it in the following way.
\begin{equation}\label{Etriangle}
A\overset{x}{\longrightarrow}B\overset{y}{\longrightarrow}C\overset{\delta}{\dashrightarrow}
\end{equation}
\item[(2)] Let $A\overset{x}{\longrightarrow}B\overset{y}{\longrightarrow}C\overset{\delta}{\dashrightarrow}$ and $A^{\prime}\overset{x^{\prime}}{\longrightarrow}B^{\prime}\overset{y^{\prime}}{\longrightarrow}C^{\prime}\overset{\delta^{\prime}}{\dashrightarrow}$ be any pair of $\EE$-triangles. If a triplet $(a,b,c)$ realizes $(a,c)\colon\delta\to\delta^{\prime}$ as in $(\ref{MorphRealize})$, then we write it as
\[
\xy
(-12,6)*+{A}="0";
(0,6)*+{B}="2";
(12,6)*+{C}="4";
(24,6)*+{}="6";
(-12,-6)*+{A^{\prime}}="10";
(0,-6)*+{B^{\prime}}="12";
(12,-6)*+{C^{\prime}}="14";
(24,-6)*+{}="16";
{\ar^{x} "0";"2"};
{\ar^{y} "2";"4"};
{\ar@{-->}^{\delta} "4";"6"};
{\ar_{a} "0";"10"};
{\ar^{b} "2";"12"};
{\ar^{c} "4";"14"};
{\ar_{x^{\prime}} "10";"12"};
{\ar_{y^{\prime}} "12";"14"};
{\ar@{-->}_{\delta^{\prime}} "14";"16"};
\endxy
\]
and call $(a,b,c)$ a {\it morphism of $\EE$-triangles}.
\end{enumerate}
\end{defn}

\begin{defn}\label{DefYoneda}
Assume $\B$ and $\mathbb{E}$ satisfy {\rm (ET1)}.
By Yoneda's lemma, any $\mathbb{E}$-extension $\delta\in\mathbb{E}(C,A)$ induces natural transformations
\[ \delta_\sharp\colon\B(-,C)\Rightarrow\mathbb{E}(-,A)\ \ \text{and}\ \ \delta^\sharp\colon\B(A,-)\Rightarrow\mathbb{E}(C,-). \]
For any $X\in\B$, these $(\delta_\sharp)_X$ and $\delta^\sharp_X$ are given as follows.
\begin{enumerate}
\item[(1)] $(\delta_\sharp)_X\colon\B(X,C)\to\mathbb{E}(X,A)\ ;\ f\mapsto f^{\ast}\delta$.
\item[(2)] $\delta^\sharp_X\colon\B(A,X)\to\mathbb{E}(C,X)\ ;\ g\mapsto g_{\ast}\delta$.
\end{enumerate}
We abbreviately denote $(\delta_\sharp)_X(f)$ and $\delta^\sharp_X(g)$ by $\delta_\sharp f$ and $\delta^\sharp g$, when there is no confusion.
\end{defn}

\begin{rem}\label{Remdelautom}
By \cite[Corollary 3.8]{NP}, for any $\EE$-triangle $A\overset{x}{\longrightarrow}B\overset{y}{\longrightarrow}C\overset{\delta}{\dashrightarrow}$ and any $\delta^{\prime}\in\EE(C,A)$, the following are equivalent.
\begin{enumerate}
\item $\mathfrak{s}(\delta)=\mathfrak{s}(\delta^{\prime})$.
\item There are automorphisms $a\in\B(A,A),c\in\B(C,C)$ satisfying $xa=x$, $cy=y$ and $\delta^{\prime}=a_{\ast} c^{\ast}\delta$.
\end{enumerate}
\end{rem}

\begin{defn}\label{DefExtClosed}
Let $\mathcal{D}\subseteq\B$ be a full additive subcategory, closed under isomorphisms. We say $\mathcal{D}$ is {\it extension-closed} if it satisfies the following condition.
\begin{itemize}
\item If a conflation $A\rightarrowtail B\twoheadrightarrow C$ satisfies $A,C\in\mathcal{D}$, then $B\in\mathcal{D}$.
\end{itemize}
For two subcategories $\mathcal{D}_1,\mathcal{D}_2\subseteq \B$, we denote by $\mathcal{D}_1\ast\mathcal{D}_2$ the subcategory which consists of the objects $X$ admitting a conflation $D_1\ra X\ta D_2$ for some $D_1\in \mathcal{D}_1$ and $D_2\in \mathcal{D}_2$. In this notation, the above condition can be written as $\mathcal{D}\ast\mathcal{D}\subseteq\mathcal{D}$.
\end{defn}

In the rest of this article, we fix an extriangulated category $(\B,\EE,\mathfrak{s})$.
The following have been shown in \cite[Propositions 3.3, 3.11, 3.15]{NP}.
\begin{fact}\label{FactExact}
Let $A\overset{x}{\longrightarrow}B\overset{y}{\longrightarrow}C\overset{\delta}{\dashrightarrow}$ be any $\EE$-triangle.
Then the following sequences of natural transformations are exact.
\[ \B(C,-)\overset{-\circ y}{\Longrightarrow}\B(B,-)\overset{-\circ x}{\Longrightarrow}\B(A,-)\overset{\delta^{\sharp}}{\Longrightarrow}\mathbb{E}(C,-)\overset{y^{\ast}}{\Longrightarrow}\mathbb{E}(B,-)\overset{x^{\ast}}{\Longrightarrow}\mathbb{E}(A,-), \]
\[ \B(-,A)\overset{x\circ-}{\Longrightarrow}\B(-,B)\overset{y\circ-}{\Longrightarrow}\B(-,C)\overset{\delta_{\sharp}}{\Longrightarrow}\mathbb{E}(-,A)\overset{x_{\ast}}{\Longrightarrow}\mathbb{E}(-,B)\overset{y_{\ast}}{\Longrightarrow}\mathbb{E}(-,C). \]
\end{fact}

\begin{fact}\label{FactNP}
The following holds.
\begin{enumerate}
\item[(1)] Let $A_1\overset{x_1}{\longrightarrow}B_1\overset{y_1}{\longrightarrow}C\overset{\delta_1}{\dashrightarrow}$ and $A_2\overset{x_2}{\longrightarrow}B_2\overset{y_2}{\longrightarrow}C\overset{\delta_2}{\dashrightarrow}$ be any pair of $\mathbb{E}$-triangles. Then there is a commutative diagram in $\B$
\begin{equation}\label{Diag_PropBaer}
\xy
(-7,21)*+{A_2}="-12";
(7,21)*+{A_2}="-14";
(-21,7)*+{A_1}="0";
(-7,7)*+{M}="2";
(7,7)*+{B_2}="4";
(-21,-7)*+{A_1}="10";
(-7,-7)*+{B_1}="12";
(7,-7)*+{C}="14";
{\ar@{=} "-12";"-14"};
{\ar_{m_2} "-12";"2"};
{\ar^{x_2} "-14";"4"};
{\ar^{m_1} "0";"2"};
{\ar^{e_1} "2";"4"};
{\ar@{=} "0";"10"};
{\ar_{e_2} "2";"12"};
{\ar^{y_2} "4";"14"};
{\ar_{x_1} "10";"12"};
{\ar_{y_1} "12";"14"};
{\ar@{}|\circlearrowright "-12";"4"};
{\ar@{}|\circlearrowright "0";"12"};
{\ar@{}|\circlearrowright "2";"14"};
\endxy
\end{equation}
which satisfies
\begin{eqnarray*}
&\mathfrak{s}(y_2^{\ast}\delta_1)=[A_1\overset{m_1}{\longrightarrow}M\overset{e_1}{\longrightarrow}B_2],&\\
&\mathfrak{s}(y_1^{\ast}\delta_2)=[A_2\overset{m_2}{\longrightarrow}M\overset{e_2}{\longrightarrow}B_1],&\\
&m_{1\ast}\delta_1+m_{2\ast}\delta_2=0.&
\end{eqnarray*}
\item[(2)] Dual of {\rm (1)}.
\end{enumerate}
\end{fact}

The following lemma is another version of \cite[Corollary 3.16]{NP}.
\begin{prop}\label{PBPO}
Let $A\overset{x}{\longrightarrow}B\overset{y}{\longrightarrow}C\overset{\delta}{\dashrightarrow}$ be any $\EE$-triangle, let $f\colon A\rightarrow D$ be any morphism, and let $D\overset{d}{\longrightarrow}E\overset{e}{\longrightarrow}C\overset{f_{\ast}\delta}{\dashrightarrow}$ be any $\EE$-triangle realizing $f_{\ast}\delta$. Then there is a morphism $g$ which gives a morphism of $\EE$-triangles
\begin{equation}\label{diag_PBPO}
\xymatrix{
A \ar[r]^{x} \ar[d]_f &B \ar[r]^{y} \ar[d]^g &C \ar@{=}[d]\ar@{-->}[r]^{\delta}&\\
D \ar[r]_{d} &E \ar[r]_{e} &C\ar@{-->}[r]_{f_{\ast}\delta}&
}
\end{equation}
and moreover, $A\overset{\svecv{-f}{x}}{\longrightarrow}D\oplus B\overset{\svech{d}{g}}{\longrightarrow}E\overset{e^{\ast}\delta}{\dashrightarrow}$ becomes an $\EE$-triangle.
\end{prop}
\begin{proof}
By Fact \ref{FactNP}, we get the following commutative diagram made of $\EE$-triangles
$$\xymatrix{
&A \ar[d]_-{\svecv{{}^{\exists}h}{x}} \ar@{=}[r] &A \ar[d]^x&\\
D \ar[r]_{\svecv{1}{0}} \ar@{=}[d] &D\oplus B \ar[r]^{\svech{0}{1}} \ar[d]^{\svech{d}{{}^{\exists}g^{\prime}}} &B \ar[d]^y\ar@{-->}[r]^0&\\
D \ar[r]_d &E \ar@{-->}[d]_{e^{\ast}\delta}\ar[r]_e &C \ar@{-->}[d]^{\delta}\ar@{-->}[r]^{f_{\ast}\delta}&\\
&&&
}
$$
satisfying $\svecv{h}{x}_{\ast}\delta+\svecv{1}{0}_{\ast}f_{\ast}\delta=0$, in which we may assume that the middle row is of the form $D\overset{\svecv{1}{0}}{\longrightarrow}D\oplus B\overset{\svech{0}{1}}{\longrightarrow}B\overset{0}{\dashrightarrow}$ since we have $y^{\ast}f_{\ast}\delta=f_{\ast}y^{\ast}\delta=0$.
In particular, $A\overset{\svecv{h}{x}}{\longrightarrow}D\oplus B\overset{\svech{d}{g^{\prime}}}{\longrightarrow}E\overset{e^{\ast}\delta}{\dashrightarrow}$ is an $\EE$-triangle. By the above equality, we have
\[ (f+h)_{\ast}\delta=\svech{1}{0}_{\ast}\Big(\svecv{h}{x}_{\ast}\delta+\svecv{1}{0}_{\ast}f_{\ast}\delta\Big)=0 \]
in $\EE(C,D)$. Thus by the exactness of $\B(B,D)\overset{-\circ x}{\longrightarrow}\B(A,D)\overset{\delta^{\sharp}}{\longrightarrow}\EE(C,D)$, there is $b\in\B(B,D)$ which gives $f+h=bx$.

For the isomorphism $i=\left(\begin{array}{cc}1&-b\\0&1\end{array}\right)\colon D\oplus B\overset{\cong}{\longrightarrow}D\oplus B$, the following diagram is commutative.
\[
\xy
(-16,0)*+{A}="0";
(3,0)*+{}="1";
(0,8)*+{D\oplus B}="2";
(0,-8)*+{D\oplus B}="4";
(-3,0)*+{}="5";
(16,0)*+{E}="6";
{\ar^{\svecv{h}{x}} "0";"2"};
{\ar^{\svech{d}{g^{\prime}}} "2";"6"};
{\ar_{\svecv{-f}{x}} "0";"4"};
{\ar_{\svech{d}{db+g^{\prime}}} "4";"6"};
{\ar^{i}_{\cong} "2";"4"};
\endxy
\]
Thus if we put $g=db+g^{\prime}$, then it gives an $\EE$-triangle $A\overset{\svecv{-f}{x}}{\longrightarrow}D\oplus B\overset{\svech{d}{g}}{\longrightarrow}E\overset{e^{\ast}\delta}{\longrightarrow}$.
Commutativity of $(\ref{diag_PBPO})$ follows from $gx=(db+g^{\prime})x=df+(dh+g^{\prime}x)=df$ and $eg=e(db+g^{\prime})=edb+eg^{\prime}=eg^{\prime}=g$.
\end{proof}

\section{Hearts of twin cotorsion pairs}\label{Section_TwinHeart}

As before, $(\B,\EE,\mathfrak{s})$ denotes an extriangulated category.

\begin{defn}\label{extriangulated}
Let $\U$ and $\V$ be full additive subcategories of $\B$ which are closed under direct summands and isomorphisms. We call $(\U,\V)$ a \emph{cotorsion pair} if it satisfies the following conditions. As below, we always require a cotorsion pair to be {\it complete}, in the sense of \cite{Ho2}.
\begin{itemize}
\item[(a)] $\EE(\U,\V)=0$.

\item[(b)] $(\U,\V)$ is {\it complete}. Namely, for any object $B\in \B$, there exist two conflations
\begin{align*}
V_B\rightarrowtail U_B\twoheadrightarrow B,\quad
B\rightarrowtail V^B\twoheadrightarrow U^B
\end{align*}
satisfying $U_B,U^B\in \U$ and $V_B,V^B\in \V$.
\end{itemize}
A cotorsion pair $(\U,\V)$ is said to be {\it rigid} if it satisfies $\U\subseteq\V$.
\end{defn}

\begin{rem}\label{C3}
For any cotorsion pair $(\U,\V)$ on $\B$, the following holds.
\begin{itemize}
\item[(a)] A morphism $f:A\to B$ factors though $\U$ if and only if $\EE(f,\V)=0$.

\item[(b)] A morphism $f:A\to B$ factors though $\V$ if and only if $\EE(\U,f)=0$.

\item[(c)] $\U$ and $\V$ are closed under extension.

\item[(d)] $\mathcal P \subseteq \U$ and $\mathcal I \subseteq \V$.

\end{itemize}
\end{rem}

\begin{defn}\label{Cctp}
A pair of cotorsion pairs $\STUV$ on $\B$ is called a \emph{twin cotorsion pair} if it satisfies $\EE(\s,\V)=0$, or equivalently $\s\subseteq \U$.

To distinguish from a twin cotorsion pair, we sometimes call cotorsion pair $(\U,\V)$ a {\it single} cotorsion pair. Remark that any cotorsion pair $(\U,\V)$ gives a twin cotorsion pair $((\U,\V),(\U,\V))$. Thus a cotorsion pair can be regarded as a special case of a twin cotorsion pair, satisfying $\s=\U$ and $\T=\V$. In this way, any argument on twin cotorsion pairs can be applied to cotorsion pairs.
\end{defn}

\begin{rem}
If $\B$ is triangulated or exact, then this definition agrees with those in \cite{N2} and \cite{L1}, respectively.
\end{rem}

\begin{defn}\label{C5}
For any twin cotorsion pair $\STUV$, put $\W=\T\cap \U$ and call it the {\it core} of $(\U,\V)$. We define as follows.
\begin{itemize}
\item[(a)] $\B^+=\Cone(\V,\W)$. Namely, $\B^+$ is defined to be the full subcategory of $\B$, consisting of objects $B$ which admits a conflation
$$V_B\rightarrowtail W_B\twoheadrightarrow B$$
where $W_B\in \W$ and $V_B\in \V$. It can be easily shown that we have $\T\subseteq\B^+$.

\item[(b)] $\B^-=\CoCone(\W,\s)$. Namely, $\B^-$ is defined to be the full subcategory of $\B$, consisting of objects $B$ which admits a conflation
$$B\rightarrowtail W^B\twoheadrightarrow S^B$$
where $W^B\in \W$ and $S^B\in \s$. It can be easily shown that we have $\U\subseteq\B^-$.
\end{itemize}
\end{defn}

\begin{defn}\label{C6}
Let $\STUV$ be a twin cotorsion pair on $\B$, and write the quotient of $\B$ by $\W$ as $\uB=\B/\W$. For any morphism $f\in \B(X,Y)$, we denote its image in $ {\uB}(X,Y)$ by $\underline f$.

For any full additive subcategory $\mathcal C$ of $\B$ containing $\W$, similarly we put $\underline{\C}=\C/\W$. This is a full subcategory of $\uB$ consisting of the same objects as $\mathcal C$.

Put $\h=\B^+\cap\B^-$. Since $\h\supseteq \W$, we obtain a full additive subcategory $\underline{\h}\subseteq\uB$, which we call the \emph{heart} of the twin cotorsion pair.
\end{defn}

\begin{rem}\label{RemUBBT}
By using Fact \ref{FactExact}, we can easily confirm $\underline{\B}(\underline{\U},\underline{\B}^+)=0$ and $\underline{\B}(\underline{\B}^-,\underline{\T})=0$.
\end{rem}

Here are some properties of $\B^+$ and $\B^-$. The following is a corollary of Proposition \ref{PBPO}.
\begin{lem}\label{Cepi}
Let $f\in\B(A,B)$ be any morphism.
\begin{enumerate}
\item If $B\in\B^+$, then there exist $W\in\W$ and $w\in\B(W,B)$ which give a deflation $\svech{f}{w}\colon A\oplus W\to B$.
\item Dually if $A\in\B^-$, then there exist $W^{\prime}\in\W$ and $w^{\prime}\in\B(A,W^{\prime})$, which gives an inflation $\svecv{f}{w^{\prime}}\colon A\to B\oplus W^{\prime}$.
\end{enumerate}
\end{lem}
\begin{proof}
{\rm (1)} Since $B\in \B^+$, it admits a conflation $V_B\ra W_B\overset{w_B}{\ta}B$ with $V_B\in\V, W_B\in\W$. By the dual of Proposition \ref{PBPO}, we get a conflation $\xymatrix{C \ar@{ >->}[r] &A\oplus W_B \ar@{->>}[r]^-{\svech{f}{-w_B}} &B}$. {\rm (2)} can be shown dually.
\end{proof}

\begin{lem}\label{CP1}
Let $\STUV$ be as before.
\begin{itemize}
\item[(a)] If $\xymatrix{A \ar@{ >->}[r]^{f} &B \ar@{->>}[r]^g &U}$ is a conflation in $\B$ with $U\in \U$, then $A\in \B^-$ implies $B\in \B^-$.

\item[(b)] If $\xymatrix{A \ar@{ >->}[r]^{f} &B \ar@{->>}[r]^g &S}$ is a conflation in $\B$ with $S\in \s$, then $B\in \B^-$ implies $A\in \B^-$. In particular, this shows $\B^-=\CoCone(\U,\s)$.

\end{itemize}
\end{lem}

\begin{proof}
(a) Since $A\in \B^-$, it admits an $\EE$-triangle $A\overset{w^A}{\longrightarrow}W^A\to S^A\dashrightarrow$, with $W^A\in \W$ and $S^A\in \s$. By $\EE(\s,\T)=0$, we have a morphism of $\EE$-triangles
$$
\xymatrix{
A \ar[r]^{w^A} \ar[d]_f &{W^A} \ar[r] \ar[d] &{S^A} \ar[d]\ar@{-->}[r]&\\
B \ar[r]_{t^B} &{T^B} \ar[r] &{S^B}\ar@{-->}[r]&}
$$
with $S^B\in\mathcal S$ and $T^B\in\T$.
Thus we obtain a commutative diagram
$$\xymatrix{
&0\ar[d]&\\
&{\EE(T^B,V)} \ar[d]_{(t^B)^{\ast}} \ar[r] &0 \ar[d]\\
0 \ar[r] &{\EE(B,V)} \ar[r]_{f^{\ast}} &{\EE(A,V)}}
$$
for any $V\in\V$, in which the bottom row and the left column are exact. This shows $\EE(T^B,\V)=0$, and thus $T^B\in\T\cap\U=\W$.

(b) Since $B\in \B^-$, there exists a conflation $\xymatrix{B \ar@{ >->}[r]^{w^B} &W^B \ar@{->>}[r] &S^B}$ with $S^B\in\s$ and $W^B\in\W$.
By {\rm (ET4)}, we get a commutative diagram made of conflations as follows.
$$\xymatrix{
A \ar@{=}[d] \ar@{ >->}[r]^f &B \ar@{ >->}[d]^{w^B} \ar@{->>}[r]^g &S \ar@{ >->}[d]\\
A \ar@{ >->}[r]  &{W^B} \ar@{->>}[r] \ar@{->>}[d] &X \ar@{->>}[d]\\
&{S^B} \ar@{=}[r] &{S^B}
}$$
We thus get $X\in \s$ since $\s\subseteq\B$ is closed under extension. This shows $A\in \B^-$.
\end{proof}

Dually, the following holds.
\begin{lem}\label{CP2}
Let $\STUV$ be as before.
\begin{itemize}
\item[(a)] If $\xymatrix{T \ar@{ >->}[r] &A \ar@{->>}[r]^{f} &B}$ is a conflation in $\B$ with $T\in \T$, then $B\in \B^+$ implies $A\in \B^+$.

\item[(b)] If $\xymatrix{V \ar@{ >->}[r] &A \ar@{->>}[r]^{f} &B}$ is a conflation in $\B$ with $V\in \V$, then $A\in \B^+$ implies $B\in \B^+$. In particular, this shows $\B^+=\Cone(\V,\T)$.

\end{itemize}
\end{lem}

\subsection{Adjoint property}

We fix a twin cotorsion pair $\STUV$.
\begin{lem}\label{LemForRef}
Let $A\overset{a}{\longrightarrow}B\overset{b}{\longrightarrow}C\overset{\mu}{\dashrightarrow}$ and $C\overset{c}{\longrightarrow}D\overset{d}{\longrightarrow}E\overset{\nu}{\dashrightarrow}$ be $\EE$-triangles. The following are equivalent.
\begin{enumerate}
\item There exists $\theta\in\EE(E,B)$ satisfying $b_{\ast}\theta=\nu$.
\item There exists $\tau\in\EE(D,A)$ satisfying $c^{\ast}\tau=\mu$.
\end{enumerate}
\end{lem}
\begin{proof}
Assume {\rm (1)}, and realize $\theta$ by $\xymatrix{B \ar[r]^{x} &X\ar[r]^{y} &E \ar@{-->}[r]^{\theta} &}$. Then by {\rm (ET4)}, we obtain the following commutative diagram made of $\EE$-triangles,
\[\xymatrix{
A \ar@{=}[d] \ar[r]^{a} &B \ar[d]^x \ar[r]^{b} &C \ar[d]^{{}^{\exists}c^{\prime}} \ar@{-->}[r]^{\mu} &\\
A \ar[r]_{x\circ a}  &X \ar[r] \ar[d]^{y} &{}^{\exists}D^{\prime} \ar[d]^{{}^{\exists}d^{\prime}}\ar@{-->}[r]^{{}^{\exists}\gamma} &\\
&E \ar@{-->}[d]^{\theta} \ar@{=}[r] &E \ar@{-->}[d]^{b_{\ast}\theta}\\
&&
}\]
satisfying $c^{\prime\ast}\gamma=\mu$. Since $b_{\ast}\theta=\nu$, there is an isomorphism $i\in\B(D,D^{\prime})$ satisfying $c^{\prime}=i\circ c$ and $d^{\prime}\circ i=d$. Then $i^{\ast}\gamma\in\EE(D,A)$ satisfies $c^{\ast}(i^{\ast}\gamma)=c^{\prime\ast}\gamma=\mu$.
\end{proof}
\begin{prop}\label{PropForRef}
Let $B\overset{c}{\longrightarrow}D\overset{d}{\longrightarrow}E\overset{\nu}{\dashrightarrow}$ be $\EE$-triangle, and let $V_B\overset{v_B}{\longrightarrow}U_B\overset{u_B}{\longrightarrow}B\overset{\lambda}{\dashrightarrow}$ be an $\EE$-triangle satisfying $U_B\in\U,V_B\in\V$. Then the following are equivalent.
\begin{enumerate}
\item $\EE(D,V)\overset{c^{\ast}}{\longrightarrow}\EE(B,V)$ is surjective for any $V\in\V$.
\item $\EE(D,V_B)\overset{c^{\ast}}{\longrightarrow}\EE(B,V_B)$ is surjective.
\item There exists $\tau\in\EE(D,V_B)$ satisfying $c^{\ast}\tau=\lambda_B$.
\item There exists $\mu\in\EE(E,U_B)$ satisfying $(u_B)_{\ast}\mu=\nu$.
\end{enumerate}
\end{prop}
\begin{proof}
$(3)\Leftrightarrow(4)$ follows from Lemma \ref{LemForRef}.
$(1)\Rightarrow(2)\Rightarrow(3)$ is obvious. Let us show $(3)\Rightarrow(1)$.

Take any $V\in\V$ and $\lambda\in\EE(B,V)$. By the exactness of $\B(V_B,V)\overset{\lambda_B^{\sharp}}{\longrightarrow}\EE(B,V)\to0$, there exists $f\in\B(V_B,V)$ satisfying $\lambda=f_{\ast}\lambda_B$. Then $f_{\ast}\tau\in\B(D,V)$ satisfies $c^{\ast}(f_{\ast}\tau)=f_{\ast}\lambda_B=\lambda$.
\end{proof}

\begin{cor}\label{CorPushRefl}
Let
\[
\xymatrix{
A \ar[r]^{x} \ar[d]_{a} &B \ar[r]^{y} \ar[d]^{b} &{C} \ar@{=}[d] \ar@{-->}[r]^{\delta}& \\
A^{\prime} \ar[r]_{x^{\prime}} &B^{\prime} \ar[r]_{y^{\prime}} &C \ar@{-->}[r]_{a_{\ast}\delta}&
}
\]
be a morphism of $\EE$-triangles. If $\EE(B,V)\overset{x^{\ast}}{\longrightarrow}\EE(A,V)$ is surjective for any $V\in\V$, then so is $\EE(B^{\prime},V)\overset{x^{\prime\ast}}{\longrightarrow}\EE(A^{\prime},V)$ for any $V\in\V$.
\end{cor}
\begin{proof}
Resolve $A$ and $A^{\prime}$ by $\EE$-triangles
\[ V_A\to U_A\overset{u_A}{\longrightarrow}A\dashrightarrow,\quad V_{A^{\prime}}\to U_{A^{\prime}}\overset{u_{A^{\prime}}}{\longrightarrow}A^{\prime}\dashrightarrow \]
satisfying $U_A,U_{A^{\prime}}\in\U$ and $V_A,V_{A^{\prime}}\in\V$. Then $a$ induces a morphism of $\EE$-triangles as follows.
\[
\xymatrix{
V_A \ar[r]^{} \ar[d]_{{}^{\exists}} &U_A \ar[r]^{u_A} \ar[d]^{{}^{\exists}u} &A \ar[d]^{a} \ar@{-->}[r]^{}& \\
V_{A^{\prime}} \ar[r]_{} &U_{A^{\prime}} \ar[r]_{u_{A^{\prime}}} &A^{\prime} \ar@{-->}[r]_{}&
}
\]
By Proposition \ref{PropForRef}, there exists $\mu\in\EE(C,U_A)$ which gives $(u_A)_{\ast}\mu=\delta$. Then $u_{\ast}\mu\in\EE(C,U_{A^{\prime}})$ satisfies
\[ (u_{A^{\prime}})_{\ast}(u_{\ast}\mu)=a_{\ast}(u_A)_{\ast}\mu=a_{\ast}\delta. \]
Again by Proposition \ref{PropForRef}, $\EE(B^{\prime},V)\overset{x^{\prime\ast}}{\longrightarrow}\EE(A^{\prime},V)$ becomes surjective for any $V\in\V$.
\end{proof}

\begin{defn}\label{DefOfRef}
\begin{itemize}
\item[(a)] A conflation with $Z\in\B^+,S\in\s$
\[ \xymatrix{B \ar@{ >->}[r]^{z} &Z \ar@{->>}[r] &S} \]
is called a {\it reflection sequence} for $B$, if $\EE(Z,V)\overset{z^{\ast}}{\longrightarrow}\EE(B,V)$ is surjective $($and thus bijective$)$ for any $V\in\V$.
\item[(b)] Dually, a conflation with $X\in\B^-,V\in\V$
\[ \xymatrix{V \ar@{ >->}[r] &X \ar@{->>}[r]^{x} &B} \]
is called a {\it coreflection sequence} for $B$, if $\EE(S,X)\overset{x_{\ast}}{\longrightarrow}\EE(S,B)$ is surjective for any $S\in\s$.
\end{itemize}
\end{defn}

\begin{prop}\label{PropFor+Adj}
If $B\overset{z}{\ra}Z\ta S$ is a reflection sequence, then the following hold for any $Y\in\B^+$.
\begin{enumerate}
\item $-\circ z\colon\B^+(Z,Y)\to\B(B,Y)$ is surjective.
\item $-\circ \underline{z}\colon\uB^+(Z,Y)\to\uB(B,Y)$ is bijective.
\end{enumerate}
Dually for coreflection sequences.
\end{prop}
\begin{proof}
By $Y\in\B^+$, there is an $\EE$-triangle $\xymatrix{V_Y \ar[r]^{v_Y} &W_{Y}\ar[r]^{w_Y} &Y \ar@{-->}[r]^{\lambda_Y} &}$ satisfying $V_Y\in\V_Y$ and $W_Y\in\W_Y$.

{\rm (1)} Let $f\in\B(B,Y)$ be any morphism. Since $z^{\ast}\colon\EE(Z,V_Y)\to\EE(B,V_Y)$ is surjective, there is $\beta\in\EE(Z,V_Y)$ satisfying $z^{\ast}\beta=f^{\ast}\lambda_Y$. Then the exactness of
\begin{eqnarray*}
&\B(Z,Y)\overset{(\lambda_Y)_{\sharp}}{\longrightarrow}\EE(Z,V_Y)\overset{(v_Y)_{\ast}}{\longrightarrow}\EE(Z,W_Y),&\\
&0\to\EE(Z,W_Y)\overset{z^{\ast}}{\longrightarrow}\EE(B,W_Y)&
\end{eqnarray*}
and $z^{\ast}(v_Y)_{\ast}\beta=f^{\ast}(v_Y)_{\ast}\lambda_Y=0$ shows the existence of $g\in\B(Z,Y)$ which gives $g^{\ast}\lambda_Y=\beta$.

By $(\lambda_Y)_{\sharp}(f-gz)=f^{\ast}\lambda_Y-z^{\ast}\beta=0$, we obtain $h\in\B(B,W_Y)$ satisfying $f-gz=w_Y h$. By the exactness of
\begin{equation}\label{PFA}
\B(Z,W_Y)\overset{-\circ z}{\longrightarrow}\B(B,W_Y)\to0,
\end{equation}
we have $i\in\B(Z,W_Y)$ which gives $iz=h$.
Then $g+w_Yi\in\B(Z,W_Y)$ satisfies $f=(g+w_Yi)z$.

{\rm (2)} Suppose $f\in\B^+(Z,Y)$ satisfies $\underline{fz}=0$. This implies that there exists $g\in\B(B,W_Y)$ satisfying $w_Yg=fz$. By the exactness of $(\ref{PFA})$, there is $h\in\B(Z,W_Y)$ which gives $hz=g$. Then by the exactness of
\[ \B(S,Y)\to\B(Z,Y)\overset{-\circ z}{\longrightarrow}\B(B,Y) \]
and the equality $(f-w_Yh)z=0$, it follows that $f-w_Yh$ factors through $S$. Since $\uB(S,Y)=0$ by Remark \ref{RemUBBT}, this means $\underline{f}=\underline{w_Yh}=0$.
\end{proof}

\begin{cor}\label{CorFor+Adj}
By Proposition \ref{PropFor+Adj}, a reflection sequence $B\overset{z}{\ra}Z\ta S$ gives a reflection $(Z,\underline{z})$ of $B$ along the inclusion functor $\uB^+\hookrightarrow\uB$, in the sense of \cite[Definition 3.1]{Bo}.
In particular, $Z\in\B^+$ is uniquely determined by $B$, up to isomorphism in $\uB^+$. Dually for coreflection sequences.
\end{cor}
\begin{proof}
This immediately follows from Proposition \ref{PropFor+Adj}.
\end{proof}

\begin{lem}\label{Lem5B2}
Let $A\in\B$ be any object, and let $B\overset{z}{\ra}Z\ta S$ be a reflection sequence. For any morphism $f\in\B(A,B)$, the following are equivalent.
\begin{enumerate}
\item $zf$ satisfies $\underline{zf}=0$.
\item $f$ factors some object in $\U$.
\end{enumerate}
\end{lem}
\begin{proof}
$(2)\Rightarrow (1)$ follows from $\uB(\underline{\U},\uB^+)=0$, as in Remark \ref{RemUBBT}. Let us show the converse.

Let $V_B\overset{v_B}{\longrightarrow}U_B\overset{u_B}{\longrightarrow}B\overset{\lambda_B}{\dashrightarrow}$ be an $\EE$-triangle satisfying $U_B\in\U,V_B\in\V$. By Proposition \ref{PropForRef}, there exists $\tau\in\EE(Z,V_B)$ satisfying $z^{\ast}\tau=\lambda_B$.
Then $\underline{zf}=0$ and $\EE(\W,V_B)=0$ shows $f^{\ast}\lambda_B=(zf)^{\ast}\tau=0$. Thus {\rm (2)} follows from the exactness of $\B(A,U_B)\overset{u_B\circ-}{\longrightarrow}\B(A,B)\overset{(\lambda_B)_{\sharp}}{\longrightarrow}\EE(A,V_B)$.
\end{proof}

The following gives a (co)reflection sequence for each $B\in\B$ (Proposition \ref{Prop+isRef}).

\begin{defn}\label{Cre}
Let $B\in\B$ be any object. We define as follows.
\begin{enumerate}
\item Take two $\EE$-triangles
\begin{align*}
V_B\to U_B\overset{u_B}{\longrightarrow} B\overset{\delta}{\dashrightarrow},\quad
U_B\to T^U\to S^U\dashrightarrow
\end{align*}
where $U_B \in \U$, $V_B\in \V$, $T^U\in \T$ and $S^U\in \s$ which implies $T^U\in\W$. By \cite[(ET4)]{NP}, we get the following commutative diagram.
\begin{equation}\label{CF1}
\xymatrix{
V_B \ar@{=}[d] \ar@{ >->}[r] &U_B \ar@{ >->}[d]^u \ar@{->>}[r]^{u_B} &B \ar@{ >->}[d]^{p_B} \ar@{-->}[r]^{\delta} &\\
V_B \ar@{ >->}[r]  &T^U \ar@{->>}[r]^t \ar@{->>}[d] &B^+ \ar@{->>}[d]\ar@{-->}[r]^{\delta^{\prime\prime}} &\\
&S^U \ar@{=}[r] &S^U
}
\end{equation}
By Definition \ref{C5}, we have $B^+\in \B^+$.
\item Dually, take the following two conflations
\begin{align*}
B\rightarrowtail T^B \twoheadrightarrow S^B,\quad
V_T\rightarrowtail U_T \twoheadrightarrow T^B
\end{align*}
where $U_T\in \U$, $V_T\in \V$, $T^B\in \T$ and $S^B\in \s$. By {\rm (ET4)$^{\mathrm{op}}$}, we get the following commutative diagram
\begin{equation}\label{CF1op}
\xymatrix{
V_T \ar@{ >->}[d] \ar@{=}[r] &{V_T} \ar@{ >->}[d]\\
B^- \ar@{->>}[d]_{m_B} \ar@{ >->}[r] &{U_T} \ar@{->>}[d] \ar@{->>}[r] &{S^B} \ar@{=}[d]\\
B \ar@{ >->}[r] &{T^B} \ar@{->>}[r] &{S^B}}
\end{equation}
in which, $B^-$ belongs to $\B^-$.
\end{enumerate}
\end{defn}

\begin{prop}\label{Prop+isRef}
For any $B\in\B$, the following holds.
\begin{enumerate}
\item[{\rm (a)}] The conflation $B\overset{p_B}{\ra}B^+\ta S^U$ in Definition \ref{Cre} {\rm (1)} is a reflection sequence for $B$.
\item[{\rm (b)}] The conflation $V_T\ra B^-\overset{m_B}{\ta}B$ in Definition \ref{Cre} {\rm (2)} is a coreflection sequence for $B$.
\end{enumerate}
\end{prop}
\begin{proof}
{\rm (a) follows from Corollary \ref{CorPushRefl}. Dually for {\rm (b)}.}
\end{proof}

\begin{lem}\label{Ceq}
For any $B\in \B$, the following are equivalent.
\begin{itemize}
\item[(a)] $B\in \U$.

\item[(b)] $B^+\in \W$. Namely, it is isomorphic to $0$ in $\uB^+$.

\item[(c)] $\underline {p_B}=0$ in $\uB$.
\end{itemize}
\end{lem}
\begin{proof}
{\rm (a)} implies $\uB^+(B^+,\uB^+)\cong\uB(B,\uB^+)=0$ by Remark \ref{RemUBBT}, which shows {\rm (b)}.

$\mathrm{(b)}\Rightarrow\mathrm{(c)}$ is obvious. $\mathrm{(c)}\Rightarrow \mathrm{(a)}$ follows from Lemma \ref{Lem5B2}. Indeed, $\underline{p_B}=0$ implies that $\id_B$ factors through some object in $\U$, which means $B\in\U$ by Remark \ref{C3}.
\end{proof}

\begin{defn}\label{Defsigmas}
$\ \ $
\begin{enumerate}
\item By Corollary \ref{CorFor+Adj} and Proposition \ref{Prop+isRef}, the inclusion functor $i^+\colon\uB\hookrightarrow\uB^+$ has a right adjoint functor $\sigma^+$. In particular, this $\sigma^+$ can be given by the following.
\begin{itemize}
\item For each $B\in\B$, choose a reflection sequence $B\overset{p_B}{\ra}B^+\ta S^U$ as in Definition \ref{Cre} {\rm (1)}, and put $\sigma^+(B)=B^+$. If $B$ itself belongs to $\B^+$, we may choose as $B^+=B$.
\item For any morphism $\underline f\colon B\rightarrow C$, we define $\sigma^+(\underline f)$ as the unique morphism which makes the following diagram commutative.
$$\xymatrix{
B \ar[r]^{\underline f} \ar[d]_{\underline {p_B}} &C \ar[d]^{\underline {p_C}}\\
B^+ \ar@{-->}[r]_{\sigma^+(\underline f)} &C^+.
}
$$
\end{itemize}
\item Dually, the inclusion $i^-\colon\uB^-\hookrightarrow\uB$ has a left adjoint functor $\sigma^-$, defined by using a coreflection sequence as in Definition \ref{Cre} {\rm (2)}.
\end{enumerate}
\end{defn}

\begin{prop}\label{8.2}
The functor $\sigma^+$ has the following properties.
\begin{itemize}
\item[(a)] $\sigma^+$ is an additive functor.
\item[(b)] $\sigma^+|_{\uB^+}=\id_{\uB^+}$.
\item[(c)] For any morphism $f\colon A\rightarrow B$, $\sigma^+(\underline f)=0$ holds in $\uB$ if and only if $f$ factors through $\U$.
\end{itemize}
Dually for $\sigma^-$.
\end{prop}

\begin{proof}
(a),(b) can be concluded easily by definition.
By the adjoint property, $\sigma^+(\underline{f})=0$ is equivalent to $\underline{p_Bf}=\sigma^+(\underline{f})\circ\underline{p_A}=0$. Thus {\rm (c)} follows from Lemma \ref{Lem5B2}.
\end{proof}

\begin{lem}\label{C14}
Let
\begin{equation}\label{C14diag}
\xymatrix{
A \ar[r]^{w} \ar[d]_{f} &{W} \ar[r] \ar[d]^g &{S} \ar@{=}[d] \ar@{-->}[r]^{\rho}& \\
B \ar[r]_{c} &{C} \ar[r] &{S} \ar@{-->}[r]_{f_{\ast}\rho}&
}
\end{equation}
be any morphism of $\EE$-triangles satisfying $W\in\W$ and $S\in\s$. Then we have the following.
\begin{enumerate}
\item  For any $X\in\B$,
\[ \B(C,X)\overset{-\circ c}{\longrightarrow}\B(B,X)\overset{\underline{-\circ f}}{\longrightarrow}\uB(A,X) \]
is exact.
\item For any $Y\in\B^+$,
\[ 0\to\uB(C,Y)\overset{-\circ\underline{c}}{\longrightarrow}\uB(B,Y)\overset{-\circ\underline{f}}{\longrightarrow}\uB(A,Y) \]
is exact.
\item $\sigma^+(A)\overset{\sigma^+(\underline{c})}{\longrightarrow}\sigma^+(B)\overset{\sigma^+(\underline{f})}{\longrightarrow}\sigma^+(C)\to0$ is a cokernel sequence in $\uB^+$.
\end{enumerate}
\end{lem}
\begin{proof}
{\rm (1)} Suppose that a morphism $x\in\B(B,X)$ satisfies $\underline{xf}=0$. This means that $xf$ factors through an object in $\W$. Thus $(xf)_{\ast}\rho=0$ follows from $\EE(S,\W)=0$. By the exactness of
\[ \B(C,X)\overset{-\circ c}{\longrightarrow}\B(B,X)\overset{(f_{\ast}\rho)^{\sharp}}{\longrightarrow}\EE(S,X), \]
we obtain a morphism $x^{\prime}\in\B(C,X)$ satisfying $x^{\prime}c=x$.

{\rm (2)} It suffices to show that $\uB(C,Y)\overset{-\circ\underline{c}}{\longrightarrow}\uB(B,Y)$ is monomorphic. Let $y\in\B(C,Y)$ be any morphism satisfying $\underline{yc}=0$. Then there is a commutative diagram
\[
\xymatrix{
B \ar[r]^{c} \ar[d]_{w_1} &C\ar[d]^y\\
W_0 \ar[r]_{w_2} &Y
}
\]
for some object $W_0\in\W$. By the exactness of $\B(C,W_0)\overset{c^{\ast}}{\longrightarrow}\B(B,W_0)\to0$, there is $w_3\in\B(C,W_0)$ which gives $w_1=w_3c$. Then $y-w_2w_3\in\B(C,Y)$ satisfies $(y-w_2w_3)c=0$, and thus it factors through $S$. Since $\uB(S,Y)=0$, it follows that $\underline{y}=\underline{y-w_2w_3}=0$.

{\rm (3)} Indeed, by the adjoint property of $\sigma^+$ and {\rm (2)},
\[ 0\to\uB^+(\sigma^+(C),Y)\overset{-\circ\sigma^+(\underline{c})}{\longrightarrow}\uB^+(\sigma^+(B),Y)\overset{-\circ\sigma^+(\underline{f})}{\longrightarrow}\uB^+(\sigma^+(A),Y) \]
becomes exact for any $Y\in\B^+$.
\end{proof}

Remark that the existence of a diagram $(\ref{C14diag})$ implies $A\in\B^-$. Conversely, for any morphism $f\colon A\to B$ from $A\in\B^-$, we can construct such a diagram in the following way.
\begin{defn}\label{CMF}
For any morphism $f\in\B(A,B)$ with $A\in\B^-$, define $C_f\in\B$ and $c_f\in\B(B,C_f)$ as follows.
\begin{itemize}
\item Take an $\EE$-triangle $A\overset{w^A}{\longrightarrow}W^A\to S^A\overset{\rho^A}{\dashrightarrow}$ and realize $f_{\ast}\rho^A\in\EE(S^A,B)$ by an $\EE$-triangle $B\overset{c_f}{\longrightarrow}C_f\overset{s}{\longrightarrow}S^A\overset{f_{\ast}\rho^A}{\dashrightarrow}$.
\end{itemize}
This gives a morphism of $\EE$-triangles, as follows. Remark that $C_f$ is unique up to isomorphism in $\B$.
\begin{equation}\label{CF2}
\xymatrix{
A \ar[r]^{w^A} \ar[d]_{f} &{W^A} \ar[r] \ar[d] &{S^A} \ar@{=}[d] \ar@{-->}[r]^{\rho^A}& \\
B \ar[r]_{c_f} &{C_f} \ar[r]_s &{S^A} \ar@{-->}[r]_{f_{\ast}\rho^A}&
}
\end{equation}
Moreover if $B\in \B^-$, then $C_f\in \B^-$ holds by Lemma \ref{CP1}.
\end{defn}

\subsection{The heart is semi-abelian}

\begin{prop}\label{C15}
For any twin cotorsion pair $\STUV$, its heart $\underline \h$ is preabelian. Namely, any morphism in $\underline{\h}$ has a kernel and a cokernel.
\end{prop}
\begin{proof}
We only show how to construct the cokernel. Dual construction gives the kernel. Let $f\in\h(A,B)$ be any morphism. Take $C_f$ and $c_f$ as in Definition \ref{CMF}. Since $B\in\h$, it follows $C_f\in\B^-$. By Proposition \ref{Prop+isRef}, there exists $p=p_{C_f}\colon C_f\rightarrow {C_f}^+$ with ${C_f}^+\in \h$. Lemma \ref{C14} {\rm (3)} shows that $A\overset{\underline{f}}{\longrightarrow}B\overset{\underline{pc_f}}{\longrightarrow}C_f^+\to 0$ is a cokernel sequence in $\underline{\B}^+$. Since all $A,B,C_f^+$ belong to $\underline\h$, this gives a cokernel sequence in $\underline\h$.
\end{proof}

\begin{cor}\label{Ceq3}
For any morphism $f\colon A\rightarrow B$ in $\h$, the following are equivalent.
\begin{itemize}
\item[(a)] $\underline f$ is an epimorphism in $\underline \h$.

\item[(b)] ${C_f}^+\in \W.$

\item[(c)] $C_f\in \U.$
\end{itemize}
\end{cor}

\begin{proof}
The equivalence of (b) and (c) is given by Lemma \ref{Ceq}.\\
By Proposition \ref{C15}, $\underline {pc_f}$ is the cokernel of $\underline f$ in $\underline \h$. The equivalence of (a) and (b) follows immediately by this argument.
\end{proof}

\begin{defn}\label{Csemi} $(${\rm cf.} \cite[Proposition 1]{R}$)$
A preabelian category $\mathcal A$ is called \emph{left semi-abelian} if in any pull-back diagram %\comm{$\langle\langle$Comment 1: I would like to avoid using Greek letters here, since it may be confusing with extensions. $\rangle\rangle$}
$$\xymatrix{
A \ar[r]^{\mathbf{a}} \ar[d]_{\mathbf{b}} &B \ar[d]^{\mathbf{c}}\\
C \ar[r]_{\mathbf{d}} &D}
$$
in $\mathcal A$, the morphism $\mathbf{a}$ is an epimorphism whenever $\mathbf{d}$ is a cokernel morphism. Dually for a \emph{right semi-abelian} category. $\mathcal A$ is called \emph{semi-abelian} if it is both left and right semi-abelian.
\end{defn}

In this section we will prove that the heart $\underline \h$ of a twin cotorsion pair is semi-ableian.
\begin{lem}\label{C4.1}
If morphism $\mathbf{b}\in{\underline \h}(B,C)$ is a cokernel of a morphism $\underline f\in {\underline \h}(A,B)$, then there is a conflation
$B\overset{b^{\prime}}{\rightarrowtail} C^{\prime}\twoheadrightarrow S$
with $S\in\s$ and $C^{\prime}\in\h$, which admits an isomorphism $\mathbf{c}\in\underline\h(C,C^{\prime})$ satisfying $\underline{b^{\prime}}=\mathbf{c}\mathbf{b}$.
\end{lem}
\begin{proof}
For a morphism $f\in\h(A,B)$, the cokernel of $\underline{f}$ is given by $\underline{pc_f}$, as in the proof of Proposition \ref{C15}. Therefore, replacing by an isomorphism in $\underline\h$, we may assume $C=C_f$ and $\mathbf{b}=\underline{pc_f}$ from the beginning. Then by its definition, we have a morphism of $\EE$-triangles $(\ref{CF2})$, and a reflection sequence $C_f\ra C_f^+\ta S^{\prime}$ with $S^{\prime}\in\s$.
By {\rm (ET4)}, we obtain the following commutative diagram made of conflations.$$\xymatrix{
B \ar@{=}[d] \ar@{ >->}[r]^{c_f} &{C_f} \ar@{ >->}[d]_{p} \ar@{->>}[r] &{S^A} \ar@{ >->}[d]\\
B \ar@{ >->}[r] &{{C_f}^+} \ar@{->>}[r] \ar@{->>}[d] &{}^{\exists}S \ar@{->>}[d]\\
&S^{\prime} \ar@{=}[r] &S^{\prime}
}
$$
The extension-closedness of $\s\subseteq\B$ shows $S\in\s$.
\end{proof}

\begin{prop}\label{Ceq2}
Let $f\in\h(A,B)$ be any morphism. If $f^{\ast}\colon\EE(B,V)\to\EE(A,V)$ is monomorphic for any $V\in\V$, then $\underline f$ is an epimorphism in $\underline \h$.
\end{prop}
\begin{proof}
Let $(\ref{CF2})$ be a morphism of $\EE$-triangles, as in Definition \ref{CMF}. Resolve $C_f$ by an $\EE$-triangle
\[ V\to U\overset{u}{\longrightarrow}C_f\overset{\lambda}{\dashrightarrow}\quad(U\in\U,V\in\V). \]
From $(\ref{CF2})$, we obtain a commutative diagram
$$\xymatrix{
0 \ar[r] &\EE(C_f,V) \ar[d] \ar[r]^{c_f^{\ast}} &\EE(B,V) \ar[d]^{f^{\ast}}\\
&0 \ar[r] &\EE(A,V).
}
$$
where the top row is exact. Since $f^{\ast}$ is monomorphic, it follows that $c_f^{\ast}=0\colon\EE(C_f,V)\to\EE(B,V)$.

Thus by the exactness of
\[ \B(B,U)\overset{u\circ-}{\longrightarrow}\B(B,C_f)\overset{\lambda_{\sharp}}{\longrightarrow}\EE(B,V) \]
and $\lambda_{\sharp}(c_f)=c_f^{\ast}\lambda=0$, there exists $g\in\B(B,U)$ satisfying $c_f=ug$. Since $\uB(U,C_f^+)=0$, it follows that $\underline{pc_f}=\underline{pug}=0$, which means $\mathrm{Cok}\underline{f}\cong0$ in $\underline{\h}$.
\end{proof}

\begin{lem}\label{C4.2}
Suppose that $X\in \B^-$ admits a conflation
$$\xymatrix{X \ar@{ >->}[r]^{x} &B \ar@{->>}[r] &U}$$
with $B\in \h$ and $U\in \U$. Then $\sigma^+(\underline{x})\in\underline\h(X^+,B)$ is an epimorphism in $\underline\h$.
\end{lem}
\begin{proof}
Let
\[ X\overset{p_X}{\longrightarrow}X^+\to S\dashrightarrow\quad(S\in\s) \]
be an $\EE$-triangle which gives a reflection sequence. By Proposition \ref{PropFor+Adj}, there exists $y\colon X^+\rightarrow B$ satisfying $yp_X=x$. This gives $\sigma^+(\underline{x})=\underline{y}$. For any $V\in\V$, from the morphism of $\EE$-triangles
\[
\xymatrix{
X \ar[r]^{p_X} \ar@{=}[d] &X^+ \ar[r] \ar[d]^{y} &S \ar[d]_{{}^{\exists}} \ar@{-->}[r]^{}&\\
X \ar[r]_x &B \ar[r]_{} &U\ar@{-->}[r]_{}&
}
\]
we obtain the following commutative diagram
\[
\xymatrix{
0\ar[r]&\EE(B,V) \ar[d]_{y^{\ast}}\ar[r]^{x^{\ast}} &\EE(X,V) \ar@{=}[d]\\
0\ar[r]&\EE(X^+,V) \ar[r]_{p_X^{\ast}} &\EE(X,V)
}
\]
in which rows are exact. Thus $y^{\ast}\colon\EE(B,V)\to\EE(X^+,V)$ is monomorphic. By Proposition \ref{Ceq2}, this means that $\underline{y}$ is epimorphic in $\underline{\h}$.
\end{proof}

\begin{lem}\label{C4.4}
Let
\begin{equation}\label{230_1}
\xymatrix{
A \ar[r]^{\mathbf{a}} \ar[d]_{\mathbf{b}} &B \ar[d]^{\mathbf{c}}\\
C \ar[r]_{\mathbf{d}} &D}
\end{equation}
be a pull-back diagram in $\underline \h$. If there exists an object $X\in \B^-$ and morphisms $x_B\colon X\rightarrow B$, $x_C\colon X\rightarrow C$ which satisfy the following conditions, then $\mathbf{a}$ is an epimorphism in $\underline \h$.
\begin{itemize}
\item[(a)] The following diagram is commutative.
\begin{equation}\label{230_2}
\xymatrix{
X \ar[r]^{\underline {x_B}} \ar[d]_{\underline {x_C}} &B \ar[d]^{\mathbf{c}}\\
C \ar[r]_{\mathbf{d}} &D
}
\end{equation}

\item[(b)] There exists a conflation $\xymatrix{X \ar@{ >->}[r]^{x_B} &B \ar@{->>}[r] &U}$ with $U\in \U$.
\end{itemize}
\end{lem}
\begin{proof}
By the adjoint property of $\sigma^+$, $(\ref{230_2})$ induces a commutative diagram
\[
\xymatrix{
X^+ \ar[r]^{\sigma^+(\underline {x_B})} \ar[d]_{\sigma^+(\underline {x_C})} &B \ar[d]^{\mathbf{c}}\\
C \ar[r]_{\mathbf{d}} &D
}
\]
in $\underline\h$.
Since $(\ref{230_1})$ is a pull-back, there exists a morphism $\mathbf{e}\colon X^+\rightarrow A$ in $\underline \h$ which makes the following diagram commute.
$$\xymatrix{
X^+ \ar@/^/[drr]^{\sigma^+(\underline {x_B})} \ar@/_/[ddr]_{\sigma^+(\underline {x_C})} \ar[dr]^{\mathbf{e}}\\
&A \ar[r]^{\mathbf{a}} \ar[d]^{\mathbf{b}} &B \ar[d]^{\mathbf{c}}\\
&C \ar[r]_{\mathbf{d}} &D}
$$
Since $\sigma^+(\underline {x_B})$ is epimorphic by Lemma \ref{C4.2}, it follows that $\mathbf{a}$ is also an epimorphism.
\end{proof}

\begin{thm}\label{C4.3}
For any twin cotorsion pair $\STUV$, its heart $\underline \h$ is semi-abelian.
\end{thm}
\begin{proof}
By duality, we only show that $\underline \h$ is left semi-abelian. Assume we are given a pull-back diagram
$$\xymatrix{
A \ar[r]^{\mathbf{a}} \ar[d]_{\mathbf{b}} &B \ar[d]^{\mathbf{c}}\\
C \ar[r]_{\mathbf{d}} &D}
$$
in $\underline \h$ where $\mathbf{d}$ is a cokernel. Let us show that $\mathbf{a}$ is an epimorphism.

By Lemma \ref{C4.1}, replacing $D$ by an isomorphism in $\underline\h$ if necessary, we can assume that there exists an inflation $d\colon C\to D$ satisfying $\mathbf{d}=\underline d$, which admits a conflation
$$\xymatrix{C \ar@{ >->}[r]^{d} &D \ar@{->>}[r] &S}$$
with $S\in \s$. Since $D\in \h$, replacing $B$ by an isomorphism as in Lemma \ref{Cepi} (1), we may assume that there exists conflation $B^{\prime}\ra B\overset{c}{\ta}D$ such that $\mathbf{c}=\underline c$. By {\rm (ET4)$^{\op}$}, we get the following commutative diagram made of conflations.

$$\xymatrix{
B^{\prime}\ar@{ >->}[d]\ar@{=}[r]&B^{\prime}\ar@{ >->}[d]&\\
{}^{\exists}X \ar@{ >->}[r]^{{}^{\exists}x_B} \ar@{->>}[d]_{{}^{\exists}x_C} &B \ar@{->>}[d]^c \ar@{->>}[r] &S \ar@{=}[d]\\
C \ar@{ >->}[r]_d &D \ar@{->>}[r] &S}
$$
Then $X\in \B^-$ follows from Lemma \ref{CP1}. Hence $\mathbf{a}$ is epimorphic in $\underline \h$ by Lemma \ref{C4.4}.
\end{proof}

\subsection{Functor to the heart}

\begin{prop}\label{8.6}
There exists an isomorphism of functors from $\uB$ to $\underline \h$
\[ \eta\colon \sigma^+ \circ \sigma^-\overset{\cong}{\Longrightarrow}\sigma^- \circ \sigma^+. \]
\end{prop}
\begin{proof}
Let $\sigma^+,\sigma^-$ be the functors defined in Definition \ref{Defsigmas}, by using $\EE$-triangles
\[ B\overset{p_B}{\longrightarrow}B^+\to S\overset{\delta}{\dashrightarrow}\quad\text{and}\quad V\to B^-\overset{m_B}{\longrightarrow}B\overset{\mu}{\dashrightarrow}\quad(S\in\s,V\in\V) \]
as in Definition \ref{Cre}. Remark that $\{ \underline{p_B}\}_{B\in\underline{\B}}$ forms a natural transformation $\id_{\uB}\Rightarrow i^+\circ\sigma^+$, which is the unit for the adjoint pair $i^+\dashv\sigma^+$. Dually, $\{\underline{m_B}\}_{B\in\underline{\B}}\colon i^-\circ\sigma^-\Rightarrow\id_{\uB}$ gives the counit of the adjoint pair $\sigma^-\dashv i^-$.

By these adjointness, it can be easily shown that $\{\underline{p_Bm_B}\}_{B\in\underline\B}\colon i^-\circ\sigma^-\Rightarrow i^+\circ\sigma^+$ induces a natural transformation $\eta\colon\sigma^+\sigma^-\Rightarrow\sigma^-\sigma^+$, in which $\eta_B\in\underline{\h}(\sigma^+\sigma^-(B),\sigma^-\sigma^+(B))$ is determined by the commutativity of the following diagram, for each $B\in\B$.
$$\xymatrix{
B^- \ar[r]^{\underline{p_Bm_B}} \ar[d]_{\underline{p_{(B^-)}}} &B^+ &\\
\sigma^+(B^-) \ar[r]_{\eta_B} &\sigma^-(B^+)\ar[u]_{\underline{m_{(B^+)}}} &}
$$
Let us show that $\eta_B$ is an isomorphism for any $B\in\B$. By the uniqueness of (co)reflection, it suffices to show that there exist an object $Q\in\h$ and morphisms $q\in\B(Q,B^+), \ell\in\B(B^-,Q)$, which satisfies the following conditions.
\begin{itemize}
\item $(Q,\underline{\ell})$ is a reflection of $B^-$ along $i^+$,
\item $(Q,\underline{q})$ is a coreflection of $B^+$ along $i^-$
\item the diagram
$$\xymatrix{
B^- \ar[r]^{\underline{p_Bm_B}} \ar[d]_{\underline{\ell}} &B^+ &\\
Q \ar@{=}[r] &Q\ar[u]_{\underline{q}} &}
$$
is commutative.
\end{itemize}

Since $p_B^{\ast}\colon\EE(B^+,V)\overset{p_B^{\ast}}{\longrightarrow}\EE(B,V)$ is an isomorphism, there is $\theta\in\EE(B^+,V)$ which gives $p_B^{\ast}\theta=\mu$. Realize $\theta$ as $V\to Q\overset{q}{\longrightarrow}B^+\overset{\theta}{\dashrightarrow}$. Then by {\rm (ET4)$^{\mathrm{op}}$}, we obtain the following commutative diagram made of $\EE$-triangles
\begin{equation}\label{ET4opforP55}
\xymatrix{
V \ar@{ >->}[d] \ar@{=}[r] &{V} \ar@{ >->}[d]&&\\
{}^{\exists}K \ar@{->>}[d]_{{}^{\exists}k} \ar@{ >->}[r]_{{}^{\exists}\ell} &Q \ar@{->>}[d]^{q} \ar@{->>}[r] &{S} \ar@{=}[d]\ar@{-->}[r]^{{}^{\exists}\nu}&\\
B \ar@{-->}[d]_{p_B^{\ast}\theta}\ar@{ >->}[r]_{p_B} &B^+ \ar@{-->}[d]^{\theta}\ar@{->>}[r] &{S}\ar@{-->}[r]_{\delta}& \\
 & & &}
\end{equation}
satisfying $k_{\ast}\nu=\delta$. Since $p_B^{\ast}\theta=\mu$, there is an isomorphism $B^-\overset{\cong}{\longrightarrow}K$ which makes
\[
\xy
(-16,0)*+{V}="0";
(3,0)*+{}="1";
(0,8)*+{B^-}="2";
(0,-8)*+{K}="4";
(-3,0)*+{}="5";
(16,0)*+{B}="6";
{\ar^{} "0";"2"};
{\ar^{m_B} "2";"6"};
{\ar_{} "0";"4"};
{\ar_{k} "4";"6"};
{\ar_{\cong} "2";"4"};
\endxy
\]
commutative. Thus we may replace $K$ and $k$ by $B^-$ and $m_B$ from the first. By Lemmas \ref{CP1} and \ref{CP2} applied to $(\ref{ET4opforP55})$, we have $Q\in\h$. Resolve $B$ as
\[ V_B\overset{v_B}{\longrightarrow}U_B\overset{u_B}{\longrightarrow}B\overset{\lambda_B}{\dashrightarrow}\quad(U_B\in\U,V_B\in\V). \]
Then by Proposition \ref{PropForRef}, there exists $\sigma\in\EE(S,U_B)$ satisfying $(u_B)_{\ast}\sigma=\delta$.

By $\EE(U_B,V)=0$, there is $t\in\B(U_B,B^-)$ satisfying $m_Bt=u_B$. Then by the exactness of
\[ 0\to\EE(S,B^-)\overset{(m_B)_{\ast}}{\longrightarrow}\EE(S,B) \]
and the equality $(m_B)_{\ast}(\nu-t_{\ast}\sigma)=(m_B)_{\ast}\nu-(u_B)_{\ast}\sigma=\delta-\delta=0$, we obtain $\nu=t_{\ast}\sigma$.
Then we have a morphism of $\EE$-triangles as follows.
\[
\xymatrix{
U_B \ar[r]^{} \ar[d]_t &{}^{\exists}X \ar[r] \ar[d] &S \ar@{=}[d] \ar@{-->}[r]^{\sigma}&\\
B^- \ar[r]_{\ell} &Q \ar[r]_{} &S\ar@{-->}[r]_{\nu}&
}
\]
By Corollary \ref{CorPushRefl}, homomorphism
$\ell_{\ast}\colon\EE(Q,V_0)\to\EE(B^-,V_0)$ becomes surjective for any $V_0\in\V$, which means that $B^-\overset{\ell}{\ra}Q\ta S$ is a reflection sequence. Dually, we can show that $V\ra Q\overset{q}{\ta}B^+$ is a coreflection sequence. This completes the proof.
\end{proof}

\begin{defn}
Let $H\colon \B\rightarrow \underline \h$ denote the functor $\sigma^- \circ \sigma^+ \circ \pi$, where $\pi\colon \B\rightarrow \uB$ is the canonical quotient functor.  By Proposition \ref{8.6}, there is a natural isomorphism of functors
$$H=\sigma^- \circ \sigma^+ \circ \pi\cong\sigma^+ \circ \sigma^- \circ \pi.$$
\end{defn}

The following are the properties of $H$ which will be used in section \ref{section_3.2}.
\begin{rem}\label{Rem_HUVzero}
By Lemma \ref{Ceq} and its dual, we have $H(\U)=H(\T)=0$.
\end{rem}

\begin{prop}\label{PropForCoh}
Let
\begin{equation}\label{ToShowKer}
\xymatrix{
V \ar[r]^{f} \ar@{=}[d] &A \ar[r]^{a} \ar[d]^g &B \ar[d]^b \ar@{-->}[r]^{}&\\
V \ar[r]_v &U \ar[r]_{u} &C\ar@{-->}[r]_{}&
}
\end{equation}
be a morphism of $\EE$-triangles. Then
\begin{equation}\label{KerSeq}
0\to H(A)\overset{H(a)}{\longrightarrow}H(B)\overset{H(b)}{\longrightarrow}H(C)
\end{equation}
is a kernel sequence in $\underline{\h}$.
\end{prop}
\begin{proof}
\noindent [Reduction to the case where $U\in\W$ and $C\in\B^+$]
Resolve $U$ by an $\EE$-triangle
\[ U\to W\to S\dashrightarrow\quad(S\in\s,W\in\T\cap\U=\W). \]
Then by {\rm (ET4)}, we have the following diagram made of $\EE$-triangles.
\begin{equation}\label{ET4ForC}
\xymatrix{
V \ar@{=}[d] \ar[r]^{} &U \ar[d] \ar[r] &C \ar[d]^{c} \ar@{-->}[r]^{} &\\
V \ar[r]_{}  &W \ar[r] \ar[d]^{} &{}^{\exists}C^{\prime} \ar[d]^{}\ar@{-->}[r]^{} &\\
&S \ar@{-->}[d]^{} \ar@{=}[r] &S \ar@{-->}[d]^{}\\
&&
}
\end{equation}
Then $C^{\prime}$ belongs to $\Cone(\V,\W)=\B^+$. By Corollary \ref{CorPushRefl}, $C\overset{c}{\ra}C^{\prime}\ta S$ is a reflection sequence. Especially, $H(c)$ becomes an isomorphism. Composing the upper half of $(\ref{ET4ForC})$ with $(\ref{ToShowKer})$, we obtain the following morphism of $\EE$-triangles.
\[
\xymatrix{
V \ar[r]^{f} \ar@{=}[d] &A \ar[r]^{a} \ar[d] &B \ar[d]^{cb} \ar@{-->}[r]^{}&\\
V \ar[r] &W \ar[r]_{} &C^{\prime}\ar@{-->}[r]_{}&
}
\]
Since $H(c)$ is an isomorphism, we may replace $(\ref{ToShowKer})$ by the above morphism. By this replacement, we may assume $U\in\W$ and $C\in\B^+$ from the beginning.

\medskip

\noindent [Reduction to the case where $U\in\W$ and $A,B,C\in\B^+$]
By the previous step, we may assume $U\in\W$ and $C\in\B^+$.
Take an $\EE$-triangle
\[ A\overset{p_A}{\longrightarrow}A^+\to S_A\dashrightarrow \quad(A^+\in\B^+,S_A\in\s) \]
which gives a reflection sequence $A\overset{p_A}{\ra}A^+\ta S_A$.
By {\rm (ET4)}, we obtain the following commutative diagram made of $\EE$-triangles.
\[
\xymatrix{
V \ar@{=}[d] \ar[r]^{f} &A \ar[d] \ar[r]^{a} &B \ar[d]^{p} \ar@{-->}[r]^{} &\\
V \ar[r]_{p_Af}  &A^+ \ar[r] \ar[d]^{} &{}^{\exists}B^{\prime} \ar[d]^{}\ar@{-->}[r]^{{}^{\exists}\theta} &\\
&S_A \ar@{-->}[d]^{} \ar@{=}[r] &S_A \ar@{-->}[d]^{}\\
&&
}
\]
Then Lemma \ref{CP2} {\rm (2)} shows $B^{\prime}\in\B^+$. Thus by Corollary \ref{CorPushRefl},
\[ B\overset{p}{\longrightarrow}B^{\prime}\to S_A \]
becomes a reflection sequence.
Since $U\in\W\subseteq\B^+$, the sequence $\B^+(A^+,U)\to\B(A,U)\to0$ is exact, and thus there is a morphism $q\in\B^+(A^+,U)$ satisfying $g=qp_A$. By {\rm (ET3)}, we obtain a morphism of $\EE$-triangles as follows.
\[
\xymatrix{
V \ar[r]^{p_Af} \ar@{=}[d] &A^+ \ar[r]^{} \ar[d]^{q} &B^{\prime} \ar[d]^{{}^{\exists}b^{\prime}} \ar@{-->}[r]^{\theta}&\\
V \ar[r]_v &U \ar[r]_{u} &C\ar@{-->}[r]_{}&
}
\]
Since $H(p_A)$ and $H(p)$ are isomorphisms, we can replace $(\ref{ToShowKer})$ by this diagram.

\medskip

By the above arguments, we may assume $U\in\W$ and $A,B,C\in\B^+$ from the beginning. Then the dual of Lemma \ref{C14} {\rm (3)} shows that
\[ 0\to\sigma^-(A)\overset{\sigma^-(\underline{a})}{\longrightarrow}\sigma^-(B)\overset{\sigma^-(\underline{b})}{\longrightarrow}\sigma^-(C) \]
is a kernel sequence in $\underline\B^-$. Since $A,B,C\in\B^+$, this means that  $(\ref{KerSeq})$ is a kernel sequence in $\underline\h$.
\end{proof}

\begin{cor}\label{CorForCoh}
Let $A\overset{a}{\ra}B\ta S$ be any conflation, with $S\in\s$. Then $H(a)$ is an epimorphism in $\underline{\h}$.
\end{cor}
\begin{proof}
Resolve $A$ by an $\EE$-triangle $A\overset{t^A}{\longrightarrow}T^A\overset{s^A}{\longrightarrow}S^A\dashrightarrow$ with $S^A\in\s,T^A\in\T$. By Fact \ref{FactNP}, we obtain the following commutative diagram made of $\EE$-trinagles.
\[
\xymatrix{
A \ar[r]^{t^A} \ar[d]_a &T^A \ar[r]^{s^A} \ar[d]^{} &S^A \ar@{=}[d] \ar@{-->}[r]^{}&\\
B \ar[d] \ar[r] &{}^{\exists}X \ar[d] \ar[r] &S^A\ar@{-->}[r]_{}&\\
S \ar@{-->}[d] \ar@{=}[r] &S \ar@{-->}[d]& &\\
&&&
}
\]
Since $\EE(S,T^A)=0$, we see that $T^A\to X\to S$ splits. Thus $X\cong T^A\oplus S$, which implies $H(X)=0$ by Remark \ref{Rem_HUVzero}. By the dual of Proposition \ref{PropForCoh} applied to the upper half of this diagram, $H(A)\overset{H(a)}{\longrightarrow}H(B)\to 0$ becomes exact.
\end{proof}

\section{Hearts of cotorsion pairs}\label{Section_Heart}

In the rest of this article, we deal with single cotorsion pairs. For the notions introduced in the previous section, we continue to use the same symbol, applied to the case $(\s,\T)=(\U,\V)$. Thus for example, we use $\W=\U\cap\V$, $\B^+=\Cone(\V,\W)$, $\B^-=\CoCone(\W,\s)$, $\h=\B^+\cap\B^-$ and $\underline\h=\h/\W$.
\subsection{The heart is abelian}

In this section we fix a cotorsion pair $(\U,\V)$. We will prove that the heart $\underline \h=\B^+\cap \B^-/\U\cap\V$ of a cotorsion pair is abelian.

\begin{lem}\label{C17}
Let $A,B\in \h$, and let
\begin{equation}\label{LCAB}
\xymatrix{C \ar@{ >->}[r]^{g} &A \ar@{->>}[r]^{f} &B}
\end{equation}
be any conflation in $\B$. If $\underline f$ is epimorphic in $\underline \h$, then $C$ belongs to $\B^-$.
\end{lem}
\begin{proof}
Let $C\overset{g}{\longrightarrow}A\overset{f}{\longrightarrow}B\dashrightarrow$ be an $\EE$-triangle which gives $(\ref{LCAB})$. Resolve $A$ by an $\EE$-triangle $A\overset{w^A}{\longrightarrow}W^A\to U^A\overset{\rho^A}{\dashrightarrow}$. By {\rm (ET4)}, we get we get following commutative diagram made of $\EE$-triangles.
\begin{equation}\label{CF4}
\xymatrix{
C \ar@{ >->}[d]_g \ar@{=}[r] &C \ar@{ >->}[d]^h\\
A \ar@{ >->}[r]^{w^A} \ar[d]_{f} &{W^A} \ar@{->>}[r] \ar[d] &{U^A} \ar@{=}[d]\\
B \ar@{ >->}[r]_{c_f} &{C_f} \ar@{->>}[r] &{U^A}.}
\end{equation}
The lower half gives $(\ref{CF2})$ in Definition \ref{CMF}. Since $\underline f$ is an epimorphism, we have $C_f\in\U$ by Corollary \ref{Ceq3}. Remark that we have assumed $(\s,\T)=(\U,\V)$. The middle column shows $C\in \B^-$.
\end{proof}

\begin{thm}\label{C18}
For any cotorsion pair $(\U,\V)$ on $\B$, its heart
$\underline \h$ is an abelian category.
\end{thm}
\begin{proof}
Since $\underline \h$ is preabelian, it remains to show the following.
\begin{itemize}
\item[(a)] If $\underline f$ is epimorphic in $\underline \h$, then $\underline f$ is a cokernel of some morphism in $\underline \h$.

\item[(b)] If $\underline f$ is monomorphic in $\underline \h$, then $\underline f$ is a kernel of some morphism in $\underline \h$.
\end{itemize}
We only show (a), since (b) is its dual. Let $\underline f\colon A\rightarrow B$ be any epimorphism in $\underline \h$. By Lemma \ref{Cepi}, we may assume that $f$ is a deflation, and thus there is a conflation $C\overset{g}{\ra}A\overset{f}{\ta}B$. Then by Lemma \ref{C17}, it follows $C\in\B^-$. Moreover by its proof, we have a commutative diagram $(\ref{CF4})$ made of $\EE$-triangles.

If we take a reflection sequence $C\overset{p_C}{\ra}C^+\ta U$ with $U\in\U=\s$, we obtain $C^+\in\h$. By Proposition \ref{PropFor+Adj}, there exists $a\in\h(C^+,A)$ which satisfies $ap_C=g$. Let us show that
\[ C^+\overset{\underline a}{\longrightarrow}A\overset{\underline f}{\longrightarrow}B\to0 \]
is a cokernel sequence in $\underline\h$. By the assumption that $\underline{f}$ is epimorphic and by the adjoint property of $\sigma^+$, it suffices to show that
\[ \underline{\h}(B,Q)\overset{-\circ\underline{f}}{\longrightarrow}\underline{\h}(A,Q)\overset{-\circ\underline{g}}{\longrightarrow}\underline{\h}(C,Q) \]
is exact for any $Q\in\h$.

Let $r\in\h(A,Q)$ be any morphism satisfying $\underline {rg}=0$. By definition $rg$ factors through some object in $\W$. Since the epimorphicity of $\underline{f}$ implies $C_f\in\U$ in $(\ref{CF4})$ by Corollary \ref{Ceq3}, this means that $rf$ factors through $h$, and thus there exists $w\in\B(W^A,Q)$ satisfying $rg=wh$. Then the exactness of
\[ \B(B,Q)\overset{-\circ f}{\longrightarrow}\B(A,Q)\overset{-\circ g}{\longrightarrow}\B(C,Q) \]
and the equality $(r-ww^A)g=wh-wh=0$ show that there exists $s\in\B(B,Q)$ which gives $r-ww^A=sf$. Thus we obtain $\underline{r}=\underline{sf}$.
\end{proof}

\subsection{Associated cohomological functor}\label{section_3.2}
As in the case of a triangulated category, we define as follows.
\begin{defn}
Let $(\B,\EE,\mathfrak{s})$ be an extriangulated category as before, and let $\mathcal A$ be an abelian category. An additive functor $F\colon\B\to \mathcal A$ is said to be {\it cohomological}, if any conflation $A\overset{x}{\ra}B\overset{y}{\ta}C$ yields an exact sequence $F(A)\overset{F(x)}{\longrightarrow}F(B)\overset{F(y)}{\longrightarrow}F(C)$ in $\mathcal A$.
\end{defn}

\begin{rem}
If $\B$ is an exact category, this means that $F$ is a half exact functor.
\end{rem}

As a corollary of Proposition \ref{PropForCoh} and Corollary \ref{CorForCoh}, we obtain the following for a single cotorsion pair.
\begin{thm}\label{8.11}
For any cotorsion pair $(\U,\V)$, the associated functor $H\colon\B\to\underline\h$ is cohomological.
\end{thm}
\begin{proof}
Let us show the exactness of
\begin{equation}\label{ShowExact}
H(A)\xrightarrow{H(x)} H(B)\xrightarrow{H(b)} H(C)
\end{equation}
in $\underline\h$, for any $\EE$-triangle $A\overset{x}{\longrightarrow}B\overset{y}{\longrightarrow}C\dashrightarrow$.
Resolve $C$ by an $\EE$-triangle $V\overset{v}{\longrightarrow}U\overset{u}{\longrightarrow} C\dashrightarrow$ with $U\in\U,V\in\V$. By Fact \ref{FactNP}, we obtain a commutative diagram made of $\EE$-triangles as follows.
\[
\xymatrix{
&V\ar[d] \ar@{=}[r]&V\ar[d]^v&\\
A \ar[r]^{x^{\prime}} \ar@{=}[d] &{}^{\exists}B^{\prime} \ar[r]^{y^{\prime}} \ar[d]^{b} &U \ar[d]^{u} \ar@{-->}[r]^{}&\\
A \ar[r]_x &B \ar@{-->}[d] \ar[r]_y &C\ar@{-->}[d]\ar@{-->}[r]&\\
&&&
}
\]
Then Proposition \ref{PropForCoh} and Corollary \ref{CorForCoh} applied to the case $\s=\U$ shows that
\[ H(B^{\prime})\overset{H(b)}{\longrightarrow}H(B)\overset{H(y)}{\longrightarrow}H(C)\ \ \text{and}\ \ H(A)\overset{H(x^{\prime})}{\longrightarrow}H(B^{\prime})\to0 \]
are exact. This shows the exactness of $(\ref{ShowExact})$.
\end{proof}

We have the following corollaries.
\begin{cor}\label{CorEpim}
For any $A,B\in\h$ and $f\in\h(A,B)$, the following are equivalent.
\begin{enumerate}
\item $\underline f$ is epimorphic in $\underline \h$.
\item There exist a conflation $A\overset{f^{\prime}}{\ra}B^{\prime}\ta U$ in $\B$ with $U\in\U,B^{\prime}\in\h$, and a morphism $b\in\h(B,B^{\prime})$ which gives isomorphism $\underline{b}\colon B\overset{\cong}{\longrightarrow}B^{\prime}$ in $\underline\h$ satisfying $\underline{bf}=\underline f^{\prime}$.
\end{enumerate}
\end{cor}
\begin{proof}
Suppose that $\underline f$ is an epimorphism in $\underline\h$. Then in diagram $(\ref{CF2})$, we have $C_f\in\U$ by Corollary \ref{Ceq3}. By Proposition \ref{PBPO}, we obtain a conflation of the form
\[ \xymatrix{A \ar@{ >->}[r]^-{\svecv{-f}{w^A}} &B\oplus W^A \ar@{->>}[r] &C_f} \quad(W^A\in\W,C_f\in\U). \]
Modifying this by an isomorphism, $\xymatrix{A \ar@{ >->}[r]^-{\svecv{f}{-w^A}} &B\oplus W^A \ar@{->>}[r] &C_f}$ also becomes a conflation.
The converse follows from Remark \ref{Rem_HUVzero} and Theorem \ref{8.11}.
\end{proof}

\begin{cor}\label{CorKEx}
Let
\[
\xymatrix{
X \ar[r]^{} \ar[d]_z &K \ar[r] \ar[d] &C \ar@{=}[d] \ar@{-->}[r]&\\
A \ar[r]_{x} &B \ar[r]_{y} &C\ar@{-->}[r]&
}
\]
be any morphism of $\EE$-triangles satisfying $H(K)=0$. Then we obtain an exact sequence $H(X)\overset{H(z)}{\longrightarrow}H(A)\overset{H(x)}{\longrightarrow}H(B)\overset{H(y)}{\longrightarrow}H(C)$ in $\underline\h$.
\end{cor}
\begin{proof}
By Proposition \ref{PBPO}, we have a conflation of the form $\xymatrix{X \ar@{ >->}[r]^-{\svecv{z}{\ast}} &A\oplus K \ar@{->>}[r]^-{\svech{x}{\ast}} &B}$, similarly as in the proof of the previous corollary. Thus the above exactness follows from Theorem \ref{8.11}.
\end{proof}

\begin{cor}\label{8.14}
For any object $B\in \B$, the following are equivalent.
\begin{itemize}
\item[(a)] $H(B)=0$.
\item[(b)] $B\in\add(\U\ast\V)$. %For any object $\Omega U\in \Omega \U$, any morphism $u\in \Hom_\B(\Omega U, B)$ factors through $\U$.
%\item[(c)] For any object $\Omega^-V\in \Omega^-\V$, any morphism $v\in \Hom_\B(B, \Omega^-V)$ factors through $\V$.
\end{itemize}
In particular, $\add(\U\ast\V)\subseteq\B$ is extension-closed.
\end{cor}
\begin{proof}
(b)$\Rightarrow$(a) follows from Remark \ref{Rem_HUVzero} and Theorem \ref{8.11}. Let us show the converse.
By Proposition \ref{PBPO} there is a conflation
\[ U_B\ra B\oplus T^U\ta B^+ \]
in the notation of $(\ref{CF1})$. By the dual of Proposition \ref{8.2}, we see that $H(B)=\sigma^-\circ\sigma^+(B)=0$ implies $B^+\in\V$, and thus $B\in \add(\U\ast\V)$.
\end{proof}

\begin{defn}\label{DefKC}
Let $\B$ and $(\U,\V)$ be as above.
\begin{enumerate}
\item We put $\K=\add(\U\ast\V)$, and call it the \emph{kernel} of $(\U,\V)$. Here, $\add$ denotes the closure under taking direct summands in $\B$.
\item We put $\C={^{\bot_1}}\K=\U\cap{^{\bot_1}}\U$, and call it the \emph{coheart} of $(\U,\V)$. Here ${^{\bot_1}}\U\subseteq\B$ denotes the subcategory of $\B$, consisting of those $X\in\B$ which satisfies $\EE(X,\U)=0$.
\end{enumerate}
\end{defn}

\subsection{Heart-equivalence}

We continue to use the notation $\sigma^{+},\sigma^{-},H,\h,\underline{\h}$ for a cotorsion pair $(\U,\V)$. In this subsection, for another cotorsion pair $(\U^{\prime},V^{\prime})$, we denote the corresponding notions for it by $\sigma^{\prime+},\sigma^{\prime-},H^{\prime},\h^{\prime},\underline{\h}^{\prime}$ to distinguish.
\begin{defn}\label{DefHeartEquiv}
Let $(\U,\V)$ and $(\U^{\prime},\V^{\prime})$ be cotorsion pairs. $($We do not require that $((\U,\V),(\U^{\prime},\V^{\prime}))$ is a twin cotorsion pair.$)$ They are said to be {\it heart-equivalent}, if there is an equivalence $E\colon\underline\h\overset{\simeq}{\longrightarrow}\underline\h^{\prime}$ which makes
\begin{equation}\label{CommBHH}
\xy
(0,8)*+{\B}="0";
(0,-10)*+{}="1";
(-10,-6)*+{\underline\h}="2";
(10,-6)*+{\underline\h^{\prime}}="4";
{\ar_{H} "0";"2"};
{\ar^{H^{\prime}} "0";"4"};
{\ar_{E} "2";"4"};
\endxy
\end{equation}
commutative up to natural isomorphism.
\end{defn}

\begin{lem}\label{LemForHeartEq}
Let $(\U,\V)$ and $(\U^{\prime},\V^{\prime})$ be arbitrary cotorsion pairs. Assume that they satisfy $H^{\prime}(\U)=0$. Then any reflection sequence with respect to $(\U,\V)$
\begin{equation}\label{RefST}
X\overset{p}{\ra}Z\ta U\quad(U\in\U)
\end{equation}
gives an isomorphism $H^{\prime}(p)\colon H^{\prime}(X)\overset{\cong}{\longrightarrow}H^{\prime}(Z)$ in $\underline\h^{\prime}$.

If we assume $H^{\prime}(\V)=0$, then the dual holds for coreflection sequences with respect to $(\U,\V)$.
\end{lem}
\begin{proof}
Let $X\overset{p}{\longrightarrow}Z\to U\overset{\nu}{\dashrightarrow}$ be an $\EE$-triangle which gives $(\ref{RefST})$. Resolve $X$ by an $\EE$-triangle
\[ V_X\overset{v_X}{\longrightarrow}U_X\overset{u_X}{\longrightarrow}X\overset{\lambda_X}{\dashrightarrow} \]
with $U_X\in\U,V_X\in\V$. By Proposition \ref{PropForRef}, there is $\mu\in\EE(U,U_X)$ satisfying $(u_X)_{\ast}\mu=\nu$.
If we realize $\mu$ as $U_X\to U_0\to U\overset{\mu}{\dashrightarrow}$, then $U_0\in\U$ follows from the extension-closedness of $\U\subseteq\B$. We have a morphism of $\EE$-triangles as follows.
\[
\xymatrix{
U_X \ar[r] \ar[d]_{u_X} &U_0 \ar[r]^{} \ar[d]_{{}^{\exists}} &U \ar@{=}[d] \ar@{-->}[r]^{\mu}&\\
X \ar[r]_{p} &Z \ar[r]_{} &U\ar@{-->}[r]_{\nu}&
}
\]
By Corollary \ref{CorKEx}, we obtain an exact sequence
\[ 0\to H^{\prime}(X)\overset{H^{\prime}(p)}{\longrightarrow}H^{\prime}(Z)\to 0 \]
in $\underline\h^{\prime}$, which means that $H^{\prime}(p)$ is an isomorphism.
\end{proof}

\begin{prop}\label{Prop_HeartEq}
Let $(\U,\V)$ and $(\U^{\prime},\V^{\prime})$ be arbitrary cotorsion pairs, and let $\K$ and $\K^{\prime}$ be their kernels, respectively. The following are equivalent.
\begin{enumerate}
\item $(\U,\V)$ and $(\U^{\prime},\V^{\prime})$ are heart-equivalent.
\item The equalities $H^{\prime}(\U)=H^{\prime}(\V)=0$ and $H(\U^{\prime})=H(\V^{\prime})=0$ are satisfied.
\item $\K=\K^{\prime}$ holds.
\end{enumerate}
\end{prop}
\begin{proof}
By Corollary \ref{8.14}, $(2)$ is equivalent to $\U,\V\subseteq\K^{\prime}$ and $\U^{\prime},\V^{\prime}\subseteq\K$. Thus $(3)\Rightarrow(2)$ is obvious, and $(2)\Rightarrow(3)$ follows from the extension-closedness of $\K$ and $\K^{\prime}$ in $\B$. Let us show $(1)\Leftrightarrow(2)$.

$(1)\Rightarrow(2)$ This immediately follows from Remark \ref{Rem_HUVzero} and the commutativity of $(\ref{CommBHH})$.

$(2)\Rightarrow(1)$ Remark that $H^{\prime}(\U)=0$ implies that the functor $H^{\prime}$ factors through $\B/(\U\cap\V)$ to give a functor $\B/(\U\cap\V)\to\underline\h^{\prime}$.
Particularly, composing with the inclusion $\underline\h\hookrightarrow\B/(\U\cap\V)$, we obtain a functor $E\colon\underline\h\to\underline\h^{\prime}$ which satisfies $E(X)=H^{\prime}(X)$ for any object $X\in\underline\h$. Similarly, $H$ induces a functor $E^{\prime}\colon\underline\h^{\prime}\to\underline\h$.

Let us show the commutativity of $(\ref{CommBHH})$. For each object $B\in\B$, take a reflection sequence
\[ B\overset{p_B}{\ra}\sigma^+(B)\ta U_B\quad(U_B\in\U) \]
and a coreflection sequence
\[ V_B\ra H(B)\overset{m_B}{\ta}\sigma^+(B)\quad(V_B\in\V) \]
with respect to $(\U,\V)$. Then for any $B,C\in\B$ and any $f\in\B(B,C)$, the functoriality of $\sigma^+$ and $\sigma^-$ gives the following commutative diagram in $\B/(\U\cap\V)$.
\[
\xymatrix{
B \ar[r]^{\underline{p_B}} \ar[d]_{\underline{f}} &\sigma^+(B)  \ar[d]^{\sigma^+(\underline{f})} &\ar[l]_{\underline{m_B}} H(B) \ar[d]^{H(f)}\\
C \ar[r]_{\underline{p_C}} &\sigma^+(C)  & \ar[l]^{\underline{m_C}} H(C)
}
\]
By Lemma \ref{LemForHeartEq}, we obtain a commutative diagram in $\underline\h^{\prime}$
\[
\xy
(-26,8)*+{H^{\prime}(B)}="0";
(0,8)*+{H^{\prime}(\sigma^+(B))}="2";
(30,8)*+{H^{\prime}(H(B))}="4";
(50,8)*+{E(H(B))}="6";
(-26,-8)*+{H^{\prime}(C)}="10";
(0,-8)*+{H^{\prime}(\sigma^+(C))}="12";
(30,-8)*+{H^{\prime}(H(C))}="14";
(50,-8)*+{E(H(C))}="16";
{\ar^(0.42){H^{\prime}(p_B)}_(0.42){\cong} "0";"2"};
{\ar_(0.46){H^{\prime}(m_B)}^(0.46){\cong} "4";"2"};
{\ar@{=} "4";"6"};
{\ar_{H^{\prime}(f)} "0";"10"};
{\ar^{} "2";"12"};
{\ar^{E(H(f))} "6";"16"};
{\ar_(0.42){H^{\prime}(p_C)}^(0.42){\cong} "10";"12"};
{\ar^(0.46){H^{\prime}(m_C)}_(0.46){\cong} "14";"12"};
{\ar@{=} "14";"16"};
\endxy
\]
in which the horizontal arrows are isomorphisms. Thus if we define $\zeta_B\in\underline\h^{\prime}(H^{\prime}(B),E(H(B)))$ by
\[ \zeta_B=(H^{\prime}(m_B))^{-1}\circ H^{\prime}(p_B)\colon H^{\prime}(B)\to E(H(B)), \]
then the above commutativity shows that $\zeta=\{ \zeta_B\}_{B\in\B}$ gives a natural isomorphism $\zeta\colon H^{\prime}\overset{\cong}{\Longrightarrow}E\circ H$. Similarly, we can show the existence of a natural isomorphism $\zeta^{\prime}\colon H\overset{\cong}{\Longrightarrow}E^{\prime}\circ H^{\prime}$.

Composing with the inclusion $\h\hookrightarrow\B$, we see that
\[
\xy
(0,8)*+{\h}="0";
(0,-10)*+{}="1";
(-18,-6)*+{\underline\h}="2";
(0,-6)*+{\underline\h^{\prime}}="3";
(18,-6)*+{\underline\h}="4";
{\ar_{\pi} "0";"2"};
{\ar_{E} "2";"3"};
{\ar_{E^{\prime}} "3";"4"};
{\ar^{\pi} "0";"4"};
\endxy
\]
is commutative up to natural isomorhism, where $\pi$ is the canonical quotient functor. Then $E^{\prime}\circ E\cong\mathrm{Id}$ follows immediately. Similarly, we obtain $E\circ E^{\prime}\cong\mathrm{Id}$.
\end{proof}

Let us write as $(\U,\V)\le (\U^{\prime},\V^{\prime})$ when $\V\subseteq\V^{\prime}$ holds, as in \cite[\S 2]{S}. Then for a fixed cotorsion pair $(\U,\V)$, Proposition \ref{Prop_HeartEq} tells that the largest {\it possible} cotorsion pair which is heart-equivalent to $(\U,\V)$ should be $(\C,\K)$. While the pair $(\C,\K)$ is not always a cotorsion pair, a necessary and sufficient condition for it to be a cotorsion pair can be given by the existence of enough projectives of certain type in $\underline\h$. We will deal with this in the next section (Theorem \ref{HCoH}).
\begin{rem}\label{RemRigid}
The following is obvious.
\begin{enumerate}
\item If $(\C,\K)$ is a cotorsion pair, it is rigid by definition.
\item If $(\U,\V)$ is rigid from the first, then $(\C,\K)=(\U,\V)$ holds. Especialy, $(\C,\K)$ is indeed a cotorsion pair in this case.
\end{enumerate}
\end{rem}

\section{Hearts with enough projectives}\label{Section_ProjectiveHeart}

In the following sections, we assume that $\B$ has enough projectives. As before, the subcategory of projectives is denoted by $\mathcal P\subseteq\B$.
In this section, we give a sufficient condition for the heart $\underline{\h}$ to have enough projectives (Theorem \ref{HCoH}), in terms of the kernel and the coheart. Moreover, under this condition, the heart admits an equivalence $\underline{\h}\simeq\mod(\C/\mathcal P)$ (Proposition \ref{Propmod}).

The existence of enough projectives gives the following criterion.
\begin{prop}\label{approxi}
Assume $\B$ has enough projectives, as above.
Let $\U,\V\subseteq\B$ be full additive subcategories, closed under taking direct summands and isomorphisms. Suppose that $\U$ satisfies the following conditions.
\begin{itemize}
\item[{\rm (i)}] $\mathcal P\subseteq\U$.
\item[{\rm (ii)}] $\U\subseteq\B$ is extension-closed.
\end{itemize}
Then, the following are equivalent.
\begin{enumerate}
\item $(\U,\V)$ is a cotorsion pair on $\B$.
\item $\EE(\U,\V)=0$ and $\B=\CoCone(\V,\U)$ holds.
\end{enumerate}
\end{prop}
\begin{proof}
It suffices to show $\B=\Cone(\V,\U)$, under the assumption of {\rm (2)}.
Let $B\in\B$ be any object. Since $\B$ has enough projectives, there is a conflation $X\ra P\ta B$ with $P\in\mathcal P$.
Since $\B=\CoCone(\V,\U)$, this $X$ has a conflation $X\ra V^{\prime}\ta U^{\prime}$ with $U^{\prime}\in\U,V^{\prime}\in\V$. By Fact \ref{FactNP}, we obtain the following commutative diagram made of conflations.
$$\xymatrix{
X \ar@{ >->}[r] \ar@{ >->}[d] &P \ar@{->>}[r] \ar@{ >->}[d] &B \ar@{=}[d]\\
V^{\prime} \ar@{ >->}[r] \ar@{->>}[d] &{}^{\exists}M \ar@{->>}[r] \ar@{->>}[d] &B\\
U^{\prime} \ar@{=}[r] &U^{\prime}
}
$$
By the assumption of {\rm (i),(ii)}, it follows $M\in\U$.
\end{proof}

\subsection{Condition for the heart to have enough projectives}
In the rest of this section, we fix a single cotorsion pair $(\U,\V)$ on $\B$.

\begin{defn}\label{DefOmegaB1}
For any subcategory $\mathcal B_1\subseteq\B$, we define as $\Omega\B_1=\CoCone(\mathcal{P},\B_1)$. Namely, $\Omega \mathcal B_1$ is the subcategory of $\B$ consisting of objects $\Omega B_1$ such that there exists a conflation
$$\Omega B_1\rightarrowtail P\twoheadrightarrow B_1$$
with  $P\in \mathcal{P}$ and $B_1\in \B_1$.
\end{defn}

If subcategory $\B_1\subseteq\B$ contains $\mathcal P$, the above definition can be described by the following.
In the rest of this section, we write the quotient of such $\B_1$ by $\mathcal P$ as $\oB_1=\B_1/\mathcal P$. This is a full subcategory of $\oB$. For any morphism $f\in\B(X,Y)$, its image in $\oB(X,Y)$ will be denoted by $\overline f$.
\begin{prop}\label{3.7}
The following correspondence gives a functor $\Omega\colon\oB\to\oB$.
\begin{enumerate}
\item For each object $B$, choose an $\EE$-triangle
\begin{equation}\label{ConfOmegaB}
X\to P\to B\dashrightarrow
\end{equation}
with $P\in\mathcal P$, and put $\Omega B=X$.
\item Let $\overline f\in\oB(B,B^{\prime})$ be any morphism. Since $P$ is projective, any representative $f\in\B(B,B^{\prime})$ of $\overline f$ induces a morphism of $\EE$-triangles
$$\xymatrix
{\Omega B \ar[r] \ar[d]_g &P \ar[r] \ar[d] &B \ar[d]^f \ar@{-->}[r]& \\
\Omega B^{\prime} \ar[r]  &P^{\prime} \ar[r]  &B^{\prime} \ar@{-->}[r]&
}$$
where the rows are the $\EE$-triangles chosen in {\rm (1)}. We put $\Omega\overline f=\overline g$. This gives a well-defined homomorphism $\oB(B,B^{\prime})\to\oB(\Omega B,\Omega B^{\prime})$.
\end{enumerate}
Moreover, this functor is uniquely determined up to natural isomorphism, independently from the choice of $\EE$-triangles $(\ref{ConfOmegaB})$.
In particular for each $B\in\B$, object $\Omega B$ is unique up to isomorphism in $\oB$.
Remark that the image $\Omega\oB_1$ of $\oB_1$ by this functor agrees with the quotient $\overline{\Omega\B_1}=(\Omega\B_1)/\mathcal P$, where $\Omega\B_1$ is as in Definition \ref{DefOmegaB1}.
\end{prop}
\begin{proof}
This can be shown in the same way as in \cite[\S 2.2]{Ha} and \cite[Proposition 2.6]{IY}.
\end{proof}

\begin{rem}\label{8.00}
Since $H(\mathcal P)=0$, the functor $H\colon\B\to\underline{\h}$ induces a functor $\overline{H}\colon\oB\to\underline{\h}$ in a natural way.
Then we have a sequence of functors $\oB\overset{\Omega}{\longrightarrow}\oB\overset{\overline H}{\longrightarrow}\underline\h$. In particular, if $B,B^{\prime}\in\B$ satisfies $B\cong B^{\prime}$ in $\oB$, then naturally $H(B)\cong H(B^{\prime})$ holds in $\underline\h$.
\end{rem}

\begin{prop}\label{PropOExact}
Let $A\overset{x}{\ra}B\overset{y}{\ta}C$ be any conflation. Then for any choice of $\Omega C$, there exists $z\in\B(\Omega C,A)$ which makes
\begin{equation}\label{HEx}
H(\Omega C)\overset{H(z)}{\longrightarrow}H(A)\overset{H(x)}{\longrightarrow}H(B)\overset{H(y)}{\longrightarrow}H(C)
\end{equation}
exact in $\underline\h$.
\end{prop}
\begin{proof}
Let $A\overset{x}{\longrightarrow}B\overset{y}{\longrightarrow}C\dashrightarrow$ be an $\EE$-triangle, and let $\Omega C\overset{p}{\longrightarrow}P\to C\dashrightarrow$ be any $\EE$-triangle satisfying $P\in\mathcal P$. Since $P$ is projective, we obtain a morphism of $\EE$-triangles as follows.
\[
\xymatrix{
\Omega C \ar@{ >->}[r]^{} \ar[d]_z &P \ar@{->>}[r] \ar[d] &C \ar@{=}[d] \ar@{-->}[r]&\\
A \ar@{ >->}[r]_{x} &B \ar@{->>}[r]_{y} &C\ar@{-->}[r]&
}
\]
Thus Corollary \ref{CorKEx} shows the exactness of $(\ref{HEx})$.
\end{proof}

Since $\mathcal{P}\subseteq \U$, we have $\Omega\U=\CoCone(\mathcal{P},\U)\subseteq\B^-$ by Lemma \ref{CP1} applied to the case $\s=\U$. Especially $\Omega\C\subseteq\B^-$ holds for the coheart $\C$. In particular, any $X\in\Omega\C$ satisfies $H(X)\cong\sigma^+(X)$ in $\underline{\h}$ by Proposition \ref{8.6}. We define as follows.
\begin{defn}
We denote by $H(\Omega\C)\subseteq\underline{\h}$ the essential image of $\Omega\C$ by the functor $H$. As above, $X\in\h$ belongs to $H(\Omega\C)$ if and only if there exist an object $X^{\prime}\in\h$ with an isomorphism $X\cong X^{\prime}$ in $\underline{\h}$ and a reflection sequence
\begin{equation}\label{ReflXp}
\Omega C\ra X^{\prime}\ta U
\end{equation}
for some $C\in\C$ and $U\in\U$.
\end{defn}

\begin{lem}\label{LemOC}
$\Omega\C\subseteq\B$ has the following properties.
\begin{enumerate}
\item For any $\Omega C\in\Omega\C$ and $B\in\B$, a morphism $f\in\B(\Omega C,B)$ factors through an object in $\mathcal P$ if and only if it factors through one in $\mathcal\U$.
\item $\oB(\Omega\overline\C,\overline\U)=0$.
\item $(\Omega\C)\cap\U=\mathcal P$.
\item For any $\Omega C\in\Omega\C$ and any conflation $A\overset{f}{\ra}B\overset{g}{\ta}U$ with $U\in\U$,
\[ \oB(\Omega C,A)\overset{\overline f\circ-}{\longrightarrow}\oB(\Omega C,B)\to 0 \]
is exact.
\end{enumerate}
\end{lem}
\begin{proof}
{\rm (1)} follows from the existence of a conflation
\begin{equation}\label{ConfOCP}
\Omega C\rightarrowtail P\twoheadrightarrow C\quad(C\in\C)
\end{equation}
and $\EE(C,\U)=0$. {\rm (2)} immediately follows from {\rm (1)}. {\rm (3)} also follows from the existence of $(\ref{ConfOCP})$.

Let us show {\rm (4)}. Take any morphism $x\in\B(\Omega C,B)$. By {\rm (2)} we have $\overline{gx}=0$, and thus there exists $P\in\mathcal P$ and $p\in\B(\Omega C,P),q\in\B(P,U)$ satisfying $qp=gx$. Since $P$ is projective, there is $r\in\B(P,B)$ which gives $q=gr$. Then, since $g\circ(x-rp)=0$, we obtain $y\in\B(\Omega C,A)$ satisfing $fy=x-rp$. Thus it follows $\overline{x}=\overline{fy}$.
\end{proof}

\begin{prop}\label{HOCProj}
Any object $X\in H(\Omega \C)$ is projective in $\underline \h$.
\end{prop}
\begin{proof}
Let us show the exactness of
\begin{equation}\label{EXA}
\underline\h(X,A)\overset{\underline f\circ-}{\longrightarrow}\underline\h(X,B)\to 0
\end{equation}
for any epimorphism $\underline f\in \underline\h(A,B)$. Replacing $X$ by an isomorphism in $\underline\h$, we may assume that there is a reflection sequence $\Omega C\overset{p}{\ra}X\ta U_X$ with $\Omega C\in\Omega\C$ and $U_X\in\U$. Similarly by Corollary \ref{CorEpim}, replacing $B$ by an isomorphism, we may assume the existence of a conflation $A\overset{f}{\ra}B\ta U$ with $U\in\U$.

Let $x\in\h(X,B)$ be any morphism. By Lemma \ref{LemOC} {\rm (4)}, we obtain $y\in\B(\Omega C,A)$ which makes
$$\xymatrix{
\Omega C \ar[r]^{\overline{p}} \ar[d]_{\overline{y}}&X \ar[d]^{\overline{x}}\\
A \ar[r]_{\overline{f}} &B
}$$
commutative in $\oB$. Applying $\overline{H}$ we obtain $H(f)H(y)=H(x)H(p)$, namely $\underline{f}\circ H(y)=\underline{x}\circ H(p)$ in $\underline\h$. Since $H(p)$ is an isomorphism, we obtain $\underline{x}=\underline{f}\circ (H(y)H(p)^{-1})$. This shows the exactness of $(\ref{EXA})$.
\end{proof}

\begin{lem}\label{LemHCoH}
For any $B\in\B$, the following are equivalent.
\begin{enumerate}
\item There exist $\Omega C\in\Omega\C$ and $d\in\B(\Omega C,B)$ which gives an epimorphism $H(d)\in\underline\h(H(\Omega C), H(B))$.
\item $B\in\CoCone(\K,\C)$.
\end{enumerate}
\end{lem}
\begin{proof}
$(2)\Rightarrow(1)$ follows from Corollary \ref{8.14} and Proposition \ref{PropOExact}. Let us show the converse. Take an $\EE$-triangle $\Omega C\to P\to C\overset{\rho}{\dashrightarrow}$ with $P\in\mathcal P$ and $C\in\C$. Then $d\in\B(\Omega C,B)$ induces a morphism of $\EE$-triangles as follows.
\[
\xymatrix{
\Omega C \ar[r]^{} \ar[d]_d &P \ar[r] \ar[d] &C \ar@{=}[d] \ar@{-->}[r]^{\rho}&\\
B \ar[r] &{}^{\exists}Y \ar[r] &C\ar@{-->}[r]_{d_{\ast}\rho}&
}
\]
By Corollary \ref{CorKEx}, $H(\Omega C)\overset{H(d)}{\longrightarrow}H(B)\to H(Y)\to 0$ becomes exact. Since $H(d)$ is an epimorphism, it follows $Y\in\K$.
\end{proof}

\begin{thm}\label{HCoH}
For any cotorsion pair $(\U,\V)$ on $(\B,\EE,\mathfrak{s})$, the following are equivalent.
\begin{enumerate}
\item Any object $B\in\underline\h$ admits an epimorphism $X\to B$ in $\underline\h$ from some $X\in H(\Omega\C)$. In particular by Proposition \ref{HOCProj}, $\underline{\h}$ has enough projectives.
\item $(\C,\K)$ is a cotorsion pair on $\B$.
\item $\B=\CoCone(\K,\C)$.
\item $\U\subseteq\Cone(\K,\C)$.
\end{enumerate}
Moreover in this case, the subcategory of projectives in $\underline\h$ agrees with $\add(H(\Omega\C))$. Here, $\add$ denotes the closure under taking direct summands in $\underline\h$.
\end{thm}
\begin{proof}
The latter part is obvious, since any projective object come to have epimorphism from some object in $H(\Omega\C)$, which necessarily splits.
$(3)\Rightarrow(1)$ follows immediately from Lemma \ref{LemHCoH} applied to each $B\in\h$.
$(2)\Leftrightarrow(3)$ follows from $\C={}^{\bot_1}\K$ and Proposition \ref{approxi}.
$(2)\Rightarrow(4)$ is trivial.

$(1)\Rightarrow(3)$
Let $B\in\B$ be any object. Take a reflection sequence
\[ B\overset{p_B}{\ra}B^+\ta U\quad(B^+\in\B^+,U\in\U) \]
and a coreflection sequence
\[ V\ra (B^+)^-\overset{m}{\ta}B^+\quad((B^+)^-\in\h,V\in\V). \]
By assumption, there exist $X\in H(\Omega\C)$ and an epimorphism $\underline x\colon X\to(B^+)^-$ in $\underline\h$. Replacing $X$ by an isomorphism in $\underline\h$, we may assume the existence of a reflection sequence
\[ \Omega C\overset{p}{\ra}X\ta U\quad(U\in\U, C\in\C). \]
By Lemma \ref{LemOC} {\rm (4)}, there is a morphism $d\in\B(\Omega C,B)$ which makes
$$\xymatrix{
\Omega C \ar[r]^{\overline{p}} \ar[d]_{\overline{d}}&X \ar[d]^{\overline{mx}}\\
B \ar[r]_{\overline{p_B}} &B^+
}$$
commutative in $\oB$. Applying $\overline{H}$, we obtain $H(p_B)H(d)=H(m)H(x)H(p)$.
Since $H(p_B),H(m),H(p)$ are isomorphisms and $H(x)=\underline{x}$ is an epimorphism, it follows that $H(d)\colon H(\Omega C)\to H(B)$ becomes an epimorphism. Thus Lemma \ref{LemHCoH} shows $B\in\CoCone(\K,\C)$.

$(4)\Rightarrow(3)$ For any $B\in\B$, take a conflation $B\ra V^B\ta U^B$ with $U^B\in\U,V^B\in\V$. By assumption, there is a conflation $K\ra C\ta U^B$ with $K\in\K,C\in\C$. By Fact \ref{FactNP}, we obtain a commutative diagram made of conflations as follows.
$$\xymatrix{
&K \ar@{=}[r] \ar@{ >->}[d] &K \ar@{ >->}[d]\\
B \ar@{=}[d] \ar@{ >->}[r] &{}^{\exists}M \ar@{->>}[r] \ar@{->>}[d] &C \ar@{->>}[d]\\
B \ar@{ >->}[r] &V^B \ar@{->>}[r] &U^B
}
$$
Theorem \ref{8.11} and Corollary \ref{8.14} shows $H(M)=0$, and thus $M\in\K$.
\end{proof}

\begin{cor}\label{CorHCoH}
If $(\U,\V)$ is a rigid cotorsion pair, then the heart $\underline\h$ has enough projectives.
\end{cor}
\begin{proof}
This immediately follows from Remark \ref{RemRigid} {\rm (2)} and Theorem \ref{HCoH}.
\end{proof}

\begin{rem}
We may also show Corollary \ref{CorHCoH} directly, specializing the argument so far to rigid cotorsion pairs. Then in the way around, Corollary \ref{CorHCoH} and Proposition \ref{Prop_HeartEq} will show that whenever $(\C,\K)$ is a cotorsion pair (which is necessarily rigid by Remark \ref{RemRigid} {\rm (1)}), the heart $\underline\h$ of $(\U,\V)$ has enough projectives. We emphasize here that Theorem \ref{HCoH} is also giving its converse, a characterization for $(\C,\K)$ to be a cotorsion pair.
\end{rem}

\subsection{Equivalence with the category of coherent functors}

The following fact gives the desired equivalence between the heart and the functor category. For example, see \cite[Corollaries 3.9, 3.10]{Be} for the detail.%\comm{$\langle\langle$Comment A: Do you have any better citation e.g. by Auslander or Auslander-Reiten? I guess that Beligiannis should be much later than them. $\rangle\rangle$}
\begin{fact}\label{FactForCor}
If an abelian category $\mathcal A$ has enough projectives, then there is an equivalence $\mathcal A \simeq \mod (\mathrm{Proj} \mathcal A)$. Here, $\mod (\mathrm{Proj} \mathcal A)$ denotes the category of coherent functors over the category of projectives $\mathrm{Proj} \mathcal A$.
\end{fact}

\begin{lem}\label{eq}
The sequence of functors
\begin{equation}
\overline\C\overset{\Omega}{\longrightarrow}\Omega\overline\C\overset{\overline H}{\longrightarrow}H(\Omega\C),
\end{equation}
which is obtained by restricting $\oB\overset{\Omega}{\longrightarrow}\oB\overset{\overline H}{\longrightarrow}\underline\h$ in Remark \ref{8.00} onto $\overline\C$, give equivalences of categories $\overline\C\overset{\simeq}{\longrightarrow}\Omega\overline\C\overset{\simeq}{\longrightarrow}H(\Omega\C)$.
\end{lem}
\begin{proof}
Both functors are essentially surjective by definition. It suffices to show that they are fully faithful.

\medskip

\noindent {\bf {\rm 1.} Faithfulness of $\overline\C\overset{\Omega}{\longrightarrow}\Omega\overline\C$}.
Suppse that $\overline f\in\overline\C(C,C^{\prime})$ satisfies $\Omega\overline f=0$. By definition, $\Omega\overline f=\overline g$ is given by a morphism of $\EE$-triangles
\begin{equation}\label{OOO}
\xymatrix
{\Omega C \ar[r]^q \ar[d]_g &P \ar[r]^p \ar[d] &C \ar[d]^f \ar@{-->}[r]^{\rho}& \\
\Omega C^{\prime} \ar[r]_{q^{\prime}}  &P^{\prime} \ar[r]_{p^{\prime}} &C^{\prime} \ar@{-->}[r]_{\rho^{\prime}}&
}
\end{equation}
where the two rows are $\EE$-triangles satisfying $P,P^{\prime}\in\mathcal P$.
If $\overline g=0$, then $g$ factors through an projective object in $\mathcal P$. Since $\EE(C,\mathcal P)=0$, this implies $f^{\ast}\rho^{\prime}=g_{\ast}\rho=0$. Thus $f$ factors through $p^{\prime}$.

\medskip

\noindent {\bf {\rm 2.} Fullness of $\overline\C\overset{\Omega}{\longrightarrow}\Omega\overline\C$}.
Let $\overline g\in(\Omega\overline\C)(\Omega C,\Omega C^{\prime})$ be any morphism. Then by $\EE(C,\mathcal P)=0$, we obtain a morphism of $\EE$-triangles as in $(\ref{OOO})$, with some $f\in\C(C,C^{\prime})$. This gives $\Omega\overline f=\overline g$.

\medskip

\noindent {\bf {\rm 3.} Faithfulness of $\Omega\overline\C\overset{\overline H}{\longrightarrow}H(\Omega\C)$}.
This follows from Proposition \ref{8.2} and Lemma \ref{LemOC} {\rm (1)}.

\medskip

\noindent {\bf {\rm 4.} Fullness of $\Omega\overline\C\overset{\overline H}{\longrightarrow}H(\Omega\C)$}.
Let $X,X^{\prime}\in H(\Omega\C)$ be any pair of objects, and let $\underline g\in \underline\h(X,X^{\prime})$ be any morphism. Replacing by isomorphisms in $\underline\h$, we may assume there exist conflations
\[ \Omega C\overset{p}{\ra}X\ta C,\quad \Omega C^{\prime}\overset{p^{\prime}}{\ra}X^{\prime}\ta C^{\prime}\qquad(C,C^{\prime}\in\C). \]
By Lemma \ref{LemOC} {\rm (4)}, there is $f\in\B(\Omega C,\Omega C^{\prime})$ which satisfies $\overline{p^{\prime}f}=\overline{gp}$. Applying $\overline H$, we obtain $H(p^{\prime})H(f)=H(g)H(p)=\underline{g}\circ H(p)$. This shows the fullness of $\overline\C\overset{\overline H}{\longrightarrow}H(\Omega\C)$.
\end{proof}

\begin{prop}\label{Propmod}
If $(\C,\K)$ is a cotorsion pair, then the heart $\underline\h$ of $(\U,\V)$ satisfies $\underline \h \simeq \mod (\C / \mathcal P)$.
\end{prop}
\begin{proof}
Since the inclusion $H(\Omega\C)\hookrightarrow H(\Omega\C)$ induces an equivalence $\mod(H(\Omega\C))\overset{\simeq}{\longrightarrow} \mod(\add (H(\Omega\C)))$, this follows from Lemma \ref{eq}, Theorem \ref{HCoH} and Fact \ref{FactForCor}.
\end{proof}

\section{Relation with $n$-cluster tilting subcategories}\label{Section_nCluster}

In a triangulated category, if $\mathcal N$ is a cluster tilting subcategory, then $(\mathcal N,\mathcal N)$ is a cotorsion pair with the coheart $\mathcal N$ and the heart $\T/\mathcal N$. We have an equivalence $\T/\mathcal N\simeq \mod {\mathcal N} \simeq \mod(\mathcal N[-1])$ (see \cite[Corollary 4.4]{KZ}).
Fix an integer $n\ge 2$ throughout this section.
We also have a similar result for any $n$-cluster tilting subcategory. If $\mathcal N$ is an $n$-cluster tilting subcategory, then by \cite[Theorem 3.1]{IY}, we have cotorsion pairs $(\U_{\ell},\V_{\ell})=(\mathcal N\ast\mathcal N[1]\ast\cdot\cdot\cdot\ast\mathcal N[\ell-1], \mathcal N[\ell-1]\ast\mathcal N[\ell]\ast\cdot\cdot\cdot\ast\mathcal N[n-2])$ for $0<\ell<n$. Since the coheart of $(\U_{\ell},\V_{\ell})$ is $\C_{\ell}={\mathcal N}$, its heart $\underline \h_{\ell}$ becomes equivalent to $\mod{\mathcal N}$. In particular, the hearts of these cotorsion pairs $(\U_{\ell},\V_{\ell})$ are equivalent.

In this section we will generalize this to an extriangulated category with enough projectives and injectives. We define $n$-cluster tilting subcategory $\M$ in $\B$ and show how cotorsion pairs are induced from $\M$.

Assume that $(\B,\EE,\mathfrak{s})$ has enough projectives and injectives, throughout this section.
\subsection{Higher extensions}

For a subcategory $\B_1\subseteq\B$, put $\Omega^{0}\B_1=\B_1$, and define $\Omega^i\B_1$ for $i>0$ inductively by
\[ \Omega^i\B_1=\Omega(\Omega^{i-1}\B_1)=\CoCone(\mathcal{P},\Omega^{i-1}\B_1). \]
We call $\Omega^{i}\B_1$ the {\it $i$-th syzygy} of $\B_1$. Dually we define the {\it $i$-th cosyzygy} $\Sigma^i\B_1$ by $\Sigma^0\B_1=\B_1$ and $\Sigma^i\B_1=\Cone(\Sigma^{i-1}\B_1,\mathcal{I})$ for $i>0$.

Let $X$ be any object in $\B$. It admits an $\EE$-triangle
\[ X\to I^0\to \Sigma X \overset{\delta^X}{\dashrightarrow}\quad (\text{ resp. } \Omega X\to P_0\to X \overset{\delta_X}{\dashrightarrow}), \]
where $I^0\in \mathcal I$ (resp. $P_0\in \mathcal P$). We can get $\EE$-triangles
\[ \Sigma^iX\to I^i\to \Sigma^{i+1}X \overset{\delta^{\Sigma^iX}}{\dashrightarrow} (\text{ resp. } \Omega^{i+1}\to P_i\to \Omega^iX \overset{\delta_{\Omega^iX}}{\dashrightarrow}), \]
for $i>0$ recursively.

\begin{lem}\label{extn}
For any $A,X\in\B$ and $\EE$-triangles
\[ A\to I^A\overset{i}{\longrightarrow}\Sigma A\overset{\iota^A}{\dashrightarrow},\quad \Omega X\overset{p}{\longrightarrow}P\to X\overset{\rho}{\dashrightarrow} \]
satisfying $I^A\in\mathcal{I}$ and $P\in\mathcal{P}$, there is an isomorphism $\varphi_{X,A}\colon\EE(\Omega X,A)\overset{\cong}{\longrightarrow}\EE(X,\Sigma A)$.
Moreover if $B\to I^B\overset{i}{\longrightarrow}\Sigma B\overset{\iota^B}{\dashrightarrow}$ is also an $\EE$-triangle with $I^B\in\mathcal{I}$, then for any $a\in\B(A,B)$,
\begin{equation}\label{Comm_extn}
\xymatrix{
\EE(\Omega X,A) \ar[r]^{\varphi_{X,A}}_{\cong} \ar[d]_{a_{\ast}}  &\EE(X,\Sigma A)  \ar[d]^{a_{1\ast}}\\
\EE(\Omega X,B) \ar[r]^{\cong}_{\varphi_{X,B}}  &\EE(X,\Sigma B)
}
\end{equation}
is commutative. Here, $a_1\colon\Sigma A\to\Sigma B$ is any morphism satisfying $a_1^{\ast}\iota^B=a_{\ast}\iota^A$. Remark that such $a_1$ exists by the injectivity of $I^B$ $($dually to {\rm (2)} in Proposition \ref{3.7}$)$.
\end{lem}

\begin{proof}
We have the following exact sequences. Especially, $\iota^A_{\sharp}$ and $\rho^{\sharp}$ are surjective.
$$\xymatrix{
\B(\Omega X,I^A) \ar[r]^{i\circ-}&\B(\Omega X,\Sigma A) \ar[r]^{\iota^A_{\sharp}} \ar@{=}[d] &\EE(\Omega X,A) \ar[r]&0&\text{exact}\\
\B(P,\Sigma A) \ar[r]_{-\circ p}&\B(\Omega X,\Sigma A)\ar[r]_{\rho^{\sharp}} &\EE(X,\Sigma A) \ar[r]&0&\text{exact}
}
$$
Since $I^A\in\mathcal{I}$, we have $\rho^{\sharp}(i\circ m)=i_{\ast}m_{\ast}\rho=0$ for any $m\in\B(\Omega X,I^A)$.
Thus there is a unique homomorphism $\varphi_{X,A}\colon\EE(\Omega X,A)\to\EE(X,\Sigma A)$ which satisfies $\rho^{\sharp}=\varphi_{X,A}\circ\iota^A_{\sharp}$. Similarly, we obtain a homomorphism $\psi_{X,A}\colon\EE(X,\Sigma A)\to\EE(\Omega X,A)$ satisfying $\psi_{X,A}\circ\rho^{\sharp}=\iota^A_{\sharp}$. By the uniqueness, this gives the inverse of $\varphi_{X,A}$.

Since we have
\begin{eqnarray*}
\varphi_{X,B}\circ a_{\ast}\circ\iota^A_{\sharp}(f)&=&\varphi_{X,B}(f^{\ast}a_{\ast}\iota^A)\ =\ \varphi_{X,B}(f^{\ast}a_1^{\ast}\iota^B)\\
&=&\varphi_{X,B}\circ\iota^B_{\sharp}(a_1 f)\ =\ \rho^{\sharp}(a_1 f)\\
&=&a_{1\ast}\circ\rho^{\sharp}(f)=a_{1\ast}\circ \varphi_{X,A}\circ\iota^A_{\sharp}(f)\end{eqnarray*}
for any $f\in\B(\Omega X,\Sigma A)$, commutativity of $(\ref{Comm_extn})$ follows from the surjectivity of $\iota^A_{\sharp}$.
\end{proof}

\begin{prop}\label{longexact}
Let $\xymatrix{A \ar@{ >->}[r]^a &B\ar@{->>}[r]^b &C}$ be a conflation. For any object $X\in\B$, we have the following long exact sequence.
\[ \cdots\to\EE(\Omega^i X,A)\overset{a_{\ast}}{\to}
\EE(\Omega^iX,B)\overset{b_{\ast}}{\to}
\EE(\Omega^iX,C)\to
\EE(\Omega^{i+1}X,A)\overset{a_{\ast}}{\to}
\EE(\Omega^{i+1}X,B)\overset{b_{\ast}}{\to}
%\EE(\Omega^{i+1}X,C)\to
\cdots\quad(i>0) \]
\end{prop}
\begin{proof}
As in Lemma \ref{extn}, take $\EE$-triangles
\[ A\overset{j}{\longrightarrow}I^A\overset{i}{\longrightarrow}\Sigma A\overset{\iota^A}{\dashrightarrow},\quad B\to I^B\to\Sigma B\overset{\iota^B}{\dashrightarrow}\quad(I^A,I^B\in\mathcal I) \]
and a morphism $a_1\in\B(\Sigma A,\Sigma B)$ satisfying $a_1^{\ast}\iota^B=a_{\ast}\iota^A$. Since $I^A$ is injective, any
$\EE$-triangle $A\overset{a}{\longrightarrow}B\overset{b}{\longrightarrow}C\dashrightarrow$ admits a morphism of $\EE$-triangles as follows.
$$\xymatrix{
A \ar[r]^a \ar@{=}[d] &B \ar[r]^b \ar[d]^{k} &C \ar[d]^c\ar@{-->}[r]&\\
A \ar[r]_{j} &I^A \ar[r]_{i} &\Sigma A\ar@{-->}[r]&
}
$$
We may modify $k$ to induce an $\EE$-triangle
\[ B\overset{\svecv{b}{k}}{\longrightarrow}C\oplus I^A\overset{\svech{-c}{i}}{\longrightarrow}\Sigma A\overset{a_{\ast}\iota^A}{\dashrightarrow} \]
by the dual of Proposition \ref{PBPO}.
Then $\EE(X,B)\overset{b_{\ast}}{\longrightarrow}\EE(X,C)\overset{c_{\ast}}{\longrightarrow}\EE(X,\Sigma A)$ becomes exact. Since the equality $a_{\ast}\iota^A=a_1^{\ast}\iota^B$ induces a morphism of $\EE$-triangles as follows,
$$\xymatrix{
B \ar[r]^{\svecv{b}{k}} \ar@{=}[d] &C\oplus I^A \ar[r]^{\svech{-c}{i}} \ar[d] &\Sigma A \ar[d]^{a_1}\ar@{-->}[r]^{a_{\ast}\iota^A}&\\
B \ar[r] &I^B \ar[r] &\Sigma B\ar@{-->}[r]_{\iota^B}&
}
$$
the same argument shows that $\EE(X,C)\overset{-c_{\ast}}{\longrightarrow}\EE(X,\Sigma A)\overset{{a_1}_{\ast}}{\longrightarrow}\EE(X,\Sigma B)$ is exact. Hence we have the following commutative diagram of exact sequences.
$$\xymatrix
{
&&&\EE(\Omega X,A) \ar[r]^{{a}_{\ast}} \ar[d]^{\cong}_{\varphi_{X,A}} &\EE(\Omega X,B) \ar[r]^{{b}_{\ast}} \ar[d]^{\varphi_{X,B}}_{\cong} &\EE(\Omega X,C) \\
\EE(X,A)\ar[r]^{a_{\ast}} &\EE(X,B)\ar[r]^{b_{\ast}} &\EE(X,C) \ar[r]^{-{c}_{\ast}} &\EE(X,\Sigma A) \ar[r]^{{a_1}_{\ast}} &\EE(X,\Sigma B)
}
$$
Replacing $X$ by $\Omega^{i}X$ recursively, we get the general case.
\end{proof}

\subsection{Cotorsion pairs induced from $n$-cluster tilting subcategories}

For convenience, we denote $\EE(X,\Sigma^{i}Y)\cong\EE(\Omega^{i}X,Y)$ by $\EE^{i+1}(X,Y)$ for $i\ge0$ in the rest.
\begin{defn}
A full subcategory $\M\subseteq\B$ is called {\it $n$-cluster tilting}, if it satisfies the following conditions.
\begin{enumerate}
\item $\mathcal{M }$ is contravariantly finite and covariantly finite in $\B$,
\item $X\in \mathcal{M}$ if and only if $\EE^i(X,\M)=0$ for any $i\in \{1,2,\ldots,{n-1}\}$,
\item $X\in \mathcal{M}$ if and only if $\EE^i(\M,X)=0$ for any $i\in \{1,2,\ldots,{n-1}\}$.
\end{enumerate}
In particular, $\M\subseteq\B$ becomes an additive subcategory closed by isomorphisms and direct summands, which contains all the projectives and all the injectives.
\end{defn}
In the remaining of this article, let $\M$ be an $n$-cluster tilting subcategory in $\B$.

\begin{defn}
For any $\ell\ge0$, we define a full subcategory $\M_{\ell}\subseteq\B$ inductively as follows.
\begin{itemize}
\item $\M_1=\M$.
\item $\M_{\ell}=\Cone(\M_{\ell-1},\M)$ for $\ell>1$.
\end{itemize}
\end{defn}

\begin{defn}
For any $m>0$, denote by $\M^{\bot_m}$ the subcategory of objects $X\in\B$ satisfying
\[ \EE^i(\M, X)=0\quad (1\le i \le m). \]
We have $\M^{\bot_{n-1}}=\M$.
\end{defn}

\begin{lem}\label{4.2}
We have the following.
\[ \M_{\ell}=\begin{cases}
\M^{\bot_{n-\ell}}& 0\le \ell<n, \\
\B& \ell\geq n.
\end{cases} \]
In particular, $\M_{\ell}\subseteq\B$ is closed under direct summands. Moreover, we have $\M_k\subseteq\M_{\ell}$ for any $0\le k\le\ell$.
\end{lem}
\begin{proof}
This follows from Proposition \ref{longexact}.
\end{proof}

\begin{defn}
For any $1\le j\le i$, put $\mathcal{Y}_{i,j}=\Sigma^{j-1}\M_{i-j}$. In particular, for any $1\le \ell\le n$, we put $\mathcal{Y}_{\ell}=\mathcal{Y}_{n,\ell}=\Sigma^{\ell-1}\M_{n-\ell}$.
\end{defn}

\begin{rem}\label{5.9}
The following holds.
\begin{enumerate}
\item For any $i\ge 1$, we have $\mathcal{Y}_{i,1}=\M_{i-1}$.
\item For any $2\le j\le i$, we have $\mathcal{Y}_{i,j}=\Sigma\mathcal{Y}_{i-1,j-1}$.
\end{enumerate}
\end{rem}

In the rest of this article, we show that the pair $(\M_{\ell},\mathcal{Y}_{\ell})$ becomes a cotorsion pair for each $1\le \ell\le n-1$.
Since we have $\B=\M_n=\Cone(\M_{n-1},\M)$, Lemma \ref{4.2} and the dual of Proposition \ref{approxi} shows that $(\M_1,\mathcal Y_1)=(\M,\M^{\bot_1})$ is a cotorsion pair.

\medskip

First, let us show that $\Y_{\ell}\subseteq\B$ is closed under direct summands.
\begin{lem}\label{LemToShowYClosed}
Let $\mathcal X\subseteq\B$ be a full additive subcategory closed by isomorphisms, containing $\mathcal I$. If $\mathcal X\subseteq\B$ is closed under direct summands and if $\EE(\mathcal I,\mathcal X)=0$ holds, then $\Sigma X\subseteq\B$ is also closed under direct summands.
\end{lem}
\begin{proof}
Let $X\overset{x}{\longrightarrow}I\overset{y}{\longrightarrow}B_1\oplus B_2\overset{\delta}{\dashrightarrow}$ be any $\EE$-triangle with $I\in\mathcal I$. It suffices to show that $X\in\mathcal X$ implies $B_1\in\Sigma\mathcal X$.
By {\rm (ET4)$^{\mathrm{op}}$}, we have a commutative diagram made of $\EE$-triangles
\begin{equation}\label{Diag_Ded_f}
\xymatrix{
X \ar@{=}[d] \ar[r]^{{}^{\exists}g} &{}^{\exists}X_1 \ar[d]^{{}^{\exists}m} \ar[r] &B_2 \ar[d]^{\svecv{0}{1}} \ar@{-->}[r]^{\svecv{0}{1}^{\ast}\delta} &\\
X \ar[r]_{x}  &I \ar[r]_{y} \ar[d]_{z=\svech{1}{0}y} &B_1\oplus B_2 \ar[d]^{\svech{1}{0}}\ar@{-->}[r]^{\delta} &\\
&B_1 \ar@{-->}[d]^{{}^{\exists}\nu} \ar@{=}[r] &B_1 \ar@{-->}[d]^{0}\\
&&
}
\end{equation}
satisfying $\svech{1}{0}^{\ast}\nu=g_{\ast}\delta$. We will show $X_1\in\mathcal X$. Realize $\svecv{1}{0}^{\ast}\delta\in\EE(B_1,X)$ as $X\overset{x^{\prime}}{\longrightarrow}Y\overset{y^{\prime}}{\longrightarrow}B_1\overset{\svecv{1}{0}^{\ast}\delta}{\dashrightarrow}$.
By Fact \ref{FactNP}, we have the following commutative diagram made of $\EE$-triangles
\[
\xymatrix{
&X_1\ar[d]^{\svecv{{}^{\exists}f}{m}} \ar@{=}[r]&X_1\ar[d]^m&\\
X \ar[r]_{\svecv{1}{0}} \ar@{=}[d] &X\oplus I \ar[r]_{\svech{0}{1}} \ar[d]^{{}^{\exists}e} &I \ar[d]^{z} \ar@{-->}[r]^{0}&\\
X \ar[r]_{x^{\prime}} &Y \ar@{-->}[d]^{y^{\prime\ast}\nu} \ar[r]_{y^{\prime}} &B_1\ar@{-->}[d]^{\nu}\ar@{-->}[r]_{\svecv{1}{0}^{\ast}\delta}&\\
&&&
}
\]
satisfying
\begin{equation}\label{Eq_Ded_f}
\svecv{1}{0}_{\ast}\svecv{1}{0}^{\ast}\delta+\svecv{f}{m}_{\ast}\nu=0,
\end{equation}
in which, we may assume that the middle row is of the form $X\overset{\svecv{1}{0}}{\longrightarrow}X\oplus I\overset{\svech{0}{1}}{\longrightarrow}I\overset{0}{\dashrightarrow}$ since we have $\EE(I,X)=0$ by assumption.

Since $(\ref{Eq_Ded_f})$ implies $\svecv{1}{0}^{\ast}\delta=-f_{\ast}\nu$, we obtain a morphism of $\EE$-triangles as follows.
$$\xymatrix{
X_1 \ar[r]^m \ar[d]_{-f} &I \ar[r]^{z} \ar[d]^{{}^{\exists}i} &B_1 \ar[d]^{\svecv{1}{0}}\ar@{-->}[r]^{\nu}&\\
X \ar[r]_{x} &I \ar[r]_y &B_1\oplus B_2\ar@{-->}[r]_{\delta}&
}
$$
Composing this with a morphism of $\EE$-triangles appearing in $(\ref{Diag_Ded_f})$, we obtain the following.
$$\xymatrix{
X_1 \ar[r]^m \ar[d]_{-gf} &I \ar[r]^{z} \ar[d]^{i} &B_1 \ar@{=}[d]\ar@{-->}[r]^{\nu}&\\
X_1 \ar[r]_m &I \ar[r]_{z} &B_1 \ar@{-->}[r]_{\nu}&
}
$$
Then $(1+gf)_{\ast}\nu=0$ follows, and thus there is $j\in\B(I,X_1)$ which gives $1+gf=jm$. This means that in the conflation
\begin{equation}\label{ConfToSplit}
X_1\overset{\svecv{f}{m}}{\ra}X\oplus I\overset{e}{\ta}Y,
\end{equation}
the inflation $\svecv{f}{m}$ has a retraction $\svech{g}{j}$, and thus $(\ref{ConfToSplit})$ splits. In particular, we have $X_1\in\mathcal X$.
\end{proof}

\begin{prop}\label{PropYClosed}
For any $1\le j\le i\le n$, subcategory $\mathcal Y_{i,j}\subseteq\B$ is closed under direct summands. In particular, $\mathcal Y_{\ell}\subseteq\B$ is closed under direct summands, for any $1\le l\le n$.
\end{prop}
\begin{proof}
We show by an induction on $j$. For $j=1$, this follows from Remarks \ref{4.2} and \ref{5.9}.

For $j>1$, suppose that we have shown for $j-1$. Let $i$ be any integer satisfying $j\le i\le n$, and let us show that $\mathcal Y_{i,j}\subseteq\B$ is closed under direct summands.
By the assumption of the induction, $\mathcal Y_{i-1,j-1}\subseteq\B$ is closed under direct summands. By definition, we have $\mathcal I\subseteq\mathcal Y_{i-1,j-1}$. Since $\EE(\mathcal M,\mathcal Y_{i-1,j-1})\cong\EE^j(\mathcal M,\mathcal M_{i-j})=0$ by Lemma \ref{4.2}, we also have $\EE(\mathcal I,\mathcal Y_{i-1,j-1})=0$. Thus Lemma \ref{LemToShowYClosed} shows that $\mathcal Y_{i,j}=\Sigma\mathcal Y_{i-1,j-1}\subseteq\B$ is closed under direct summands.
\end{proof}

\begin{lem}\label{4.1}
For any $i,j\ge0$, we have
\begin{equation}\label{Evanish}
\EE^k(\M_i,\M_j)=0\quad(i\le k\le n-j).
\end{equation}
In particular, $\EE(\M_{\ell},\mathcal Y_{\ell})=0$ holds for any $1\le \ell\le n$.
\end{lem}
\begin{proof}
Let us show by an induction on $i$. For $i=0$, this is trivial. For $i=1$, this follows from Lemma \ref{4.2}.

For $i>1$, suppose $(\ref{Evanish})$ holds for $i-1$, for any $j\ge 0$. Let $X\in\M_i$ be any object, and take a conflation
\[ M_{i-1}\ra M\ta X \]
with $M\in\M, M_{i-1}\in\M_{i-1}$. Then, since $\EE^{k-1}(M_{i-1},B)\to\EE^{k}(X,B)\to\EE^{k}(M,B)$ is exact for any $B\in\B$, the equality $\EE^k(M,\M_j)=0 \ (1\le k\le n-j)$ and the assumption of the induction
\[ \EE^{k-1}(M_{i-1},\M_j)=0\quad(i-1\le k-1\le n-j) \]
shows $\EE^k(X,\M_j)=0$ for any $i\le k\le n-j$.
\end{proof}

\begin{lem}\label{con}
For any $1\le j\le i$, we have $\M_i\subseteq\Cone(\mathcal{Y}_{i,j},\M_j)$.
Especially for $i=n$,
\[ \B=\M_{n}=\Cone(\mathcal{Y}_{\ell},\M_{\ell}) \]
holds for any $1\le l\le n-1$.
\end{lem}
\begin{proof}
We show by an induction on $j$. If $j=1$, then for any $i\ge 1$, we have $\M_i=\Cone(\M_{i-1},\M)=\Cone(\mathcal{Y}_{i,1},\M_1)$ by definition.

For $j>1$, suppose that we have shown for $j-1$. Let $i$ be any integer satisfying $j\le i$, and let us show $\M_i\subseteq\Cone(\mathcal{Y}_{i,j},\M_j)$.
Let $M_i\in\M_i$ be any object. By definition, there is a conflation
\begin{equation}\label{ConfMMM}
M_{i-1}\ra M\ta M_i\quad(M\in\M,M_{i-1}\in\M_{i-1})
\end{equation}
By the assumption of the induction, we have $\M_{i-1}\subseteq\Cone(\mathcal{Y}_{i-1,j-1},\M_{j-1})$. Thus there exists a conflation
\[ Y\ra M_{j-1}\ta M_{i-1}\quad(Y\in\mathcal{Y}_{i-1,j-1},M_{j-1}\in\M_{j-1}). \]
If we resolve $Y$ by a conflation $Y\ra I\ta Y^{\prime}$ with $I\in\mathcal{I}$, then we have $Y^{\prime}\in\Sigma\mathcal{Y}_{i-1,j-1}=\mathcal{Y}_{i,j}$.
By Fact \ref{FactNP}, we have the following diagram made of conflations.
\[
\xymatrix{
Y \ar@{ >->}[r] \ar@{ >->}[d] &M_{j-1} \ar@{->>}[r] \ar@{ >->}[d]^{} &M_{i-1} \ar@{=}[d]\\
I \ar@{->>}[d] \ar@{ >->}[r] &{}^{\exists}N \ar@{->>}[d] \ar@{->>}[r] &M_{i-1}\\
Y^{\prime} \ar@{=}[r] &Y^{\prime}&
}
\]
Since the middle row splits, we may assume $N=I\oplus M_{i-1}$. From $(\ref{ConfMMM})$, we have a conflation
\[ I\oplus M_{i-1}\ra I\oplus M\ta M_i. \] Then by {\rm (ET4)}, we obtain the following commutative diagram made of conflations.
\[
\xymatrix{
M_{j-1} \ar@{=}[d] \ar@{ >->}[r] &I\oplus M_{i-1} \ar@{ >->}[d] \ar@{->>}[r] &Y^{\prime} \ar@{ >->}[d] \\
M_{j-1} \ar@{ >->}[r]  &I\oplus M \ar@{->>}[r] \ar@{->>}[d]^{} &{}^{\exists}N^{\prime} \ar@{->>}[d]^{}\\
&M_i\ar@{=}[r] &M_i
}
\]
Since $I\oplus M\in\M$, it follows $N^{\prime}\in\Cone(\M_{j-1},\M)=\M_j$, and thus $M_i\in\Cone(\mathcal{Y}_{i,j},\M_j)$.
\end{proof}

\begin{lem}\label{coreMY}
$\M_{\ell}\cap \mathcal Y_{\ell}=\Sigma^{\ell-1}\M$ holds for any $1\le \ell\le n-1$.
\end{lem}
\begin{proof}
$\M_{\ell}\cap\mathcal Y_{\ell}\supseteq\Sigma^{\ell-1}\M$ follows immediately from the definition. Let us show $\M_{\ell}\cap\mathcal Y_{\ell}\subseteq\Sigma^{\ell-1}\M$ by an induction on $\ell$. For $\ell=1$, this follows from $\M\cap\mathcal Y_1=\M\cap\M_{n-1}=\M$.

For $\ell>1$, suppose that we have shown for $\ell-1$. Let $N\in\M_{\ell}\cap\mathcal Y_{\ell}$ be any object. Since $N\in\mathcal Y_{\ell}=\Sigma^{\ell-1}\M_{n-\ell}$, there is a conflation
\begin{equation}\label{ConfZIN}
Z\ra I\ta N
\end{equation}
with $Z\in\Sigma^{\ell-2}\M_{n-\ell}\subseteq\mathcal Y_{\ell-1}$ and $I\in\mathcal I$.

By Lemma \ref{4.2}, we have $\EE(\M,Z)=0$. Moreover, since $(\ref{ConfZIN})$ gives isomorphisms
\[ \EE^k(B,N)\cong\EE^{k+1}(B,Z)\quad({}^{\forall}B\in\B) \]
for any $k\ge 1$, we also obtain $\EE^k(\M,Z)=0$ for $2\le k\le n-\ell+1$ again by Lemma \ref{4.2}. Thus it follows $Z\in\M^{\bot_{n-\ell+1}}=\M_{\ell-1}$, and thus $Z\in\M_{\ell-1}\cap\mathcal Y_{\ell-1}=\Sigma^{\ell-2}\M$ by the assumption of the induction. Then $(\ref{ConfZIN})$ shows $N\in\Sigma^{\ell-1}\M$.
\end{proof}

\begin{thm}\label{nCT}
Let $\M\subseteq \B$ be any $n$-cluster tilting subcategory. Then it induces a sequence of cotorsion pairs $(\M_1,\mathcal Y_1)\ge(\M_2,\mathcal Y_2)\ge\cdots\ge(\M_{n-1},\mathcal Y_{n-1})$.
Each cotorsion pair $(\M_{\ell},\mathcal Y_{\ell})$ has the core $\W_{\ell}=\Sigma^{\ell-1}\M$, the coheart $\C_{\ell}=\M$ and the kernel $\K_{\ell}=\M_{n-1}$.
\end{thm}
\begin{proof}
By Lemma \ref{4.2}, Proposition \ref{PropYClosed}, Lemmas \ref{4.1}, \ref{con} and the dual of Proposition \ref{approxi}, the pair $(\M_{\ell},\mathcal Y_{\ell})$ becomes a cotorsion pair for any $1\le l\le n-1$. Its core is given by Lemma \ref{coreMY}.
Since we already know ${}^{\bot_1}\M_{n-1}=\M$, it remains to show $\add(\M_{\ell}\ast\mathcal Y_{\ell})=\M_{n-1}$. For $\ell=1$, this is obvious.

Let $2\le\ell\le n-1$ be any integer. Then $\M_{n-1}\supseteq\add(\M_{\ell}\ast\mathcal Y_{\ell})$ follows from Lemma \ref{4.2} and
\[ \EE(\M,\M_{\ell})=0,\quad\EE(\M,\mathcal Y_{\ell})\cong\EE^{\ell}(\M,\M_{n-\ell})=0. \]
Conversely, let $M_{n-1}\in\M_{n-1}$ be any object. By Lemma \ref{con}, there is an $\EE$-triangle
\[ Y\to M_{\ell-1}\to M_{n-1}\dashrightarrow \]
with $Y\in\mathcal Y_{n-1,l-1}$, $M_{\ell-1}\in\M_{\ell-1}$. Resolve $Y$ by an $\EE$-triangle $Y\to I\to \Sigma Y\dashrightarrow$ with $I\in\mathcal I$. Then we obtain a morphism of $\EE$-triangles
$$\xymatrix{
Y \ar[r] \ar@{=}[d] &M_{\ell-1} \ar[r] \ar[d] &M_{n-1} \ar[d]\ar@{-->}[r]&\\
Y \ar[r] &I \ar[r] &\Sigma Y \ar@{-->}[r]&
}
$$
since $I$ is injective. By the dual of Proposition \ref{PBPO}, we obtain a conflation $M_{\ell-1}\ra M_{n-1}\oplus I\ta \Sigma Y$, which means $M_{n-1}\in\add(M_{\ell-1}\ast\Sigma\mathcal Y_{n-1,l-1})\subseteq\add(\M_{\ell}\ast\mathcal Y_{\ell})$.

\end{proof}

By Proposition \ref{Prop_HeartEq} and Theorem \ref{HCoH}, we get the following corollary.
\begin{cor}\label{CornCT}
By Proposition \ref{Prop_HeartEq}, all these $(\M_1,\mathcal Y_1),\ldots,(\M_{n-1},\mathcal Y_{n-1})$ are mutually heart-equivalent. Their hearts are equivalent to $\mod (\M/\mathcal P)$ by Theorem \ref{HCoH}.
\end{cor}

\end{document}